\documentclass[a4paper,11pt,fleqn]{article}

\usepackage{amsfonts, amsmath, amssymb,amsthm}
\usepackage{graphicx}
\usepackage{xcolor}
\usepackage{algorithm}%
\usepackage{algorithmicx}%
\usepackage{algpseudocode}%
\usepackage[skip=0.5\baselineskip]{caption}
\usepackage{xcolor}         %
\usepackage{fullpage}
\usepackage{booktabs}
\usepackage{float}

\usepackage{mathrsfs} 
\usepackage{eufrak,euscript}
\usepackage{color}
\usepackage[shortlabels]{enumitem}
\usepackage{bm}
\usepackage{charter}
\definecolor{labelkey}{rgb}{0,0.08,0.45}
\definecolor{refkey}{rgb}{0,0.6,0.0}
\definecolor{Brown}{rgb}{0.45,0.0,0.05}
\definecolor{dgreen}{rgb}{0.00,0.49,0.00}
\definecolor{dblue}{rgb}{0,0.08,0.75}
\RequirePackage[colorlinks,hyperindex]{hyperref} %
\hypersetup{linktocpage=true,citecolor=dblue,linkcolor=dgreen}
\providecommand{\nor}[1]{\left\lVert {#1} \right\rVert}
\def\argmin{\operatornamewithlimits{arg\,min}}
\providecommand{\scalT}[2]{\left\langle{#1},{#2}\right\rangle}

\newcommand{\Nystrom}{\text{Nystr\"om}}

\def\argmin{\operatornamewithlimits{arg\,min}}

\DeclareMathOperator{\E}{\mathbb{E}}

\newtheorem{theorem}{Theorem}
\newtheorem{corollary}{Corollary}
\newtheorem{lemma}{Lemma}
\newtheorem{proposition}{Proposition}
\newtheorem{remark}{Remark}

\newtheorem{assumption}{Assumption}

\providecommand{\nor}[1]{\left\lVert {#1} \right\rVert}

\providecommand{\scalT}[2]{\left\langle{#1},{#2}\right\rangle}

\newcommand{\R}{\mathbb R}

\newcommand{\N}{\mathbb N}

\title{{\sffamily Stochastic Zeroth Order Descent with Structured Directions}}

\author{Marco Rando\thanks{Malga - DIBRIS, University of Genova, IT
		({\tt marco.rando@edu.unige.it}, {\tt lorenzo.rosasco@unige.it}).}
	\and Cesare Molinari\thanks{MaLGa - DIMA, University of Genova, Italy 
		({\tt molinari@dima.unige.it}, {\tt silvia.villa@unige.it}).}
	\and Silvia Villa\footnotemark[2]
	\and Lorenzo Rosasco\footnotemark[1] \thanks{Istituto Italiano di Tecnologia, Genova, Italy and CBMM - MIT, Cambridge, MA, USA}
}

\DeclareMathOperator{\diag}{diag}

\date{}

\begin{document}

\maketitle

\begin{abstract}
\noindent We introduce and analyze Structured Stochastic Zeroth order Descent (S-SZD), a finite difference approach that approximates a stochastic gradient on a set of $l\leq d$ orthogonal directions, where $d$ is the dimension of the ambient space. These directions are randomly chosen and may change at each step. For smooth convex functions we prove almost sure convergence of the iterates and a convergence rate on the function values of the form $O( (d/l) k^{-c})$ for every $c<1/2$, which is arbitrarily close to the one of Stochastic Gradient Descent (SGD) in terms of number of iterations \cite{garrigos2023handbook}. Our bound  shows the benefits of using $l$ multiple directions instead of one. For non-convex functions satisfying the Polyak-{\L}ojasiewicz condition, we establish the first convergence rates for stochastic structured zeroth order algorithms under such an assumption. We corroborate our theoretical findings with numerical simulations where the assumptions are satisfied and on the real-world problem of hyper-parameter optimization in machine learning, achieving competitive practical performance.    
\end{abstract}

\vspace{1ex}

\section{Introduction}\label{sec1}

A  problem common in economics, robotics, statistics, machine learning, and many more fields, is the minimization of functions for which no analytical form is accessible or the explicit calculation
of the gradient is too expensive, see e.g. \cite{Salimans2017EvolutionSA,10.5555/3326943.3327109,pmlr-v80-choromanski18a,flaxman2005online,spall,Conn2009IntroductionTD}. In this context, only function evaluations are available, often corrupted by some noise which is typically stochastic. This problem is called stochastic zeroth order or derivative-free optimization, and a number of different approaches have been proposed to compute  solutions, see for example \cite{Conn2009IntroductionTD,sto_md,nest,stars,zoro,ZO-BCD,pmlr-v162-gasnikov22a} and references therein. Next, we recall two main classes of approaches for which convergence guarantees can be proved, namely direct search and finite-difference methods. Direct search methods are iterative procedures that, at every step, compute a set of directions (\textit{polling set}) and update the current solution estimate by choosing a direction where the function decreases (descending direction) \cite{comp_search}.	To derive  convergence properties, the polling set generated must be a positive spanning set, which implies that the cardinality of the set is at least $d + 1$. Hence, in the worst case, to select a descending direction, $d+1$  evaluations of the function are required at each step. Recently, the idea of considering random directions~\cite{prob_ds, prob_ds_sketched} has been studied to reduce the cardinality of the polling set, and the corresponding number of function evaluations. Direct search methods have also been analyzed in stochastic \cite{ds_noise,conv_dds,Dzahini2022,dzahini2024directsearchstochasticoptimization,Audet2021} and non-smooth \cite{price2006direct,8161093,popovic2004direct} settings. %
The second class of algorithms is based on mimicking first-order optimization strategies (such as stochastic gradient descent), but replacing gradients by finite difference approximations in random directions ~\cite{nest,sto_md,ghadimi_lam,ZO-BCD}. These works consider objective functions in different settings (e.g. convex, non-convex, smooth, non-smooth) but do not use structured directions. Moreover, the stochastic setting is faced considering bounded variance assumption (see e.g. \cite{ghadimi_lam}) which is not satisfied in many scenarios (e.g. \cite[Proposition 1]{khaled2020better}). Structured finite-difference methods have been recently considered in deterministic setting with no noise \cite{kozak2021zeroth,Kozk2021ASS,stief_zeroth}. Stochastic setting has been considered in \cite{berahas2022theoretical} using a particular structured estimator built with a number of directions $l$ equal to the dimension of the domain $d$. However, the authors considered a specific noise model. Recently, also non-smooth analysis has been provided \cite{ozd} but only in deterministic setting. In this paper, we extend these results by proposing and analyzing S-SZD (Structured Stochastic Zeroth Order Descent), a derivative-free algorithm for stochastic zeroth order optimization that approximates the gradient with finite differences on a set of $l \leq d$ structured directions. Moreover, instead of relying on the bounded variance assumption, we consider the more general \eqref{eq:ABG} condition \cite{khaled2020better} providing the first analysis of a structured finite-difference method under such an assumption. We theoretically investigate the performance of S-SZD in convex and non-convex settings. Our main contributions are deriving convergence rates on the function values of order $\mathcal{O}((d/l) k^{-c})$ with $c < 1/2$ for smooth convex functions, where $k$ is the iteration counter and providing the convergence of iterates. Note that this can be arbitrarily close (depending on $c$) to the rate achieved by Stochastic Gradient Descent (SGD), which is $O(k^{-1/2})$ \cite{garrigos2023handbook}. While this dependence on the dimension $d/l$ is not optimal, the  convergence rate $\mathcal{O}((\sqrt{d/l}) k^{-c})$ can be achieved after at most $d/l$ iterations with a bigger step-size. Our approach allows us to consider one or more possible structured directions, and the  theoretical results clearly show the benefit of using more than one direction at once. Specifically, we observe that the corresponding  bound is better than similar bounds for deterministic direct search \cite{Dodangeh2016,3f5c9a64536d47be8f46be0b369d8c51}. Moreover, in our analysis, we also provided the convergence of iterates that was not considered in previous works \cite{nest,duchi_smoothing,flaxman2005online,stars} in stochastic setting. Further, we consider  non-convex objective functions satisfying the Polyak-{\L}ojasiewicz condition \cite{bolte}. In this case, we prove convergence and corresponding rates  for S-SZD. To the best of our knowledge these are the first results for  stochastic structured zeroth order under these  assumptions.  Our theoretical findings are illustrated   by numerical results that confirm the benefits of the proposed approach on both simulated data and hyper-parameter tuning in machine learning.

\noindent The paper is organized as follows. In Section \ref{sec:problem_setting} we 
discuss the problem setting and the proposed algorithm, S-SZD.  In Section \ref{sec:analysis} we state and discuss our main  results. In Section \ref{sec:experiements}, we present  numerical experiments and give some final remarks in Section \ref{sec:conclusions}.

\section{Problem Setting \& Algorithm}\label{sec:problem_setting}
Given a probability space $(\Omega, \mathcal{F}, \mathbb{P})$ and a measurable space $\mathcal{Z}$,  consider a random variable $Z:\  \Omega \to \mathcal{Z}$ and a function $F:\ \R^d \times \mathcal{Z}\rightarrow \R$. Define the function
\begin{equation*}
	f(x) := \E_Z [F(x, Z)]
\end{equation*}
and consider the  problem
\begin{equation}\label{eqn:problem}
	\min\limits_{x \in \mathbb{R}^d} \  f(x).   
\end{equation}
We assume the following set of hypothesis on the problem. In the next, we denote by $\nabla$ the gradient with respect to the $x$ variable and $f^*:=\inf f$.
\begin{assumption}[Smoothness and ABG condition]\label{ass:lam_smooth}
\ \\
\begin{itemize}
    \item[(i)] There exists $\lambda>0$ such that, for every $z\in\mathcal{Z}$, the function $x\mapsto F(x,z)$ is differentiable and its gradient is $\lambda$-Lipschitz continuous.
    \item[(ii)] The function $f$ verifies $f(x)<+\infty$ for every $x\in \R^d$, it is differentiable and has at least a minimizer.
    \item[(iii)] There exist constants $A, B, G> 0$ such that, for every $x\in\R^d$, 	\begin{equation}
  \label{eq:ABG}\tag{ABG}  \mathbb{E}_Z \Big[ \nor{\nabla F(x, Z)}^2 \Big] \leq A (f(x) - f^*) + B \| \nabla f(x) \|^2 + G.
	\end{equation} 
\end{itemize}
\end{assumption}
\noindent Note that in general Assumption~\eqref{ass:lam_smooth}(i) does not imply the differentiability of $f$.
In Lemma~\ref{lem:lip} in Appendix~\ref{app:aux} we show that, under mild additional assumptions on $F$, Assumption~\ref{ass:lam_smooth}(i) implies not only that $f$ is differentiable, but also that $\nabla F(x,Z)$ is an unbiased estimator of $\nabla f(x)$, namely that
\begin{equation}\label{unbiased}
\mathbb{E}_Z[\nabla F(x,Z)] = \nabla f(x),
\end{equation}
and that $\nabla f$ is $\lambda$-Lipschitz. In particular, if Assumption~\eqref{ass:lam_smooth} holds, Lemma~\ref{lem:lip} implies both \eqref{unbiased} and $\lambda$-Lipschitz continuity of $\nabla f$. The \eqref{eq:ABG} condition has been introduced in \cite{khaled2020better} and generalizes several different conditions usually considered in the literature to prove convergence of stochastic first order algorithms \cite{sto_md,balasubramanian2018zeroth}.
\begin{remark} The \eqref{eq:ABG} condition implies a bound on the variance of the stochastic estimator of $\nabla f$, that is
\begin{equation*}
    \E_Z[\| \nabla F(x, Z) - \nabla f(x) \|^2] \leq A (f(x) - f^*) + (B-1) \| \nabla f(x) \|^2 + G.
\end{equation*}
\end{remark}

\noindent Next, we propose a stochastic zeroth order iterative algorithm to numerically approximate a solution to problem~\eqref{eqn:problem}. Given a sequence $(Z_k)_{k\in\N}$ of independent copies of $Z$, at every step $k\in\N$, the stochastic gradient at $Z_k$ is approximated with finite differences computed along a set of $l \leq d$ structured directions $(p_k^{(i)})_{i=1}^{l}$. These directions are collected by the columns of random matrices $P_k \in \R^{d \times l}$, each one satisfying the following structural assumption.
\begin{assumption}[Orthonormal directions]\label{ass:pk_exp}
	Given a probability space $(\Omega^\prime, \mathcal{F}', \mathbb{P}')$ and $1 \leq l \leq d$, $P : \Omega^\prime \rightarrow \R^{d \times l}$ is a random variable such that
	\begin{equation*}
		\mathbb{E}_{P}[PP^\intercal] = I \quad \text{and} \quad 		P^\intercal P \stackrel{\text{a.s.}}{=} \frac{d}{l}I.
	\end{equation*}
\end{assumption}
\noindent The first condition means that in expectation the projection $PP^\intercal$ is the identity, while the second one means that the directions $(p_i)$ are almost surely orthogonal and with prescribed norm $\sqrt{d/l}$. 
\noindent Examples of  matrices satisfying this assumption are described in Appendix~\ref{app:dir_gen},  and include  coordinate and spherical directions. Given $P \in \R^{d \times l}$ satisfying Assumption~\ref{ass:pk_exp} and a finite difference discretization parameter $h > 0$,  define $D_{(P, h)} F(x, Z) \in \R^l$ as
\begin{equation*}
	[D_{(P, h)} F(x, Z)]_i := \frac{F(x + hp^{(i)},Z) - F(x,Z)}{h} \qquad i=1,...,l,
\end{equation*}
where $p^{(i)}\in\R^d$ is the $i$-th column of $P$. Then,  consider the following approximation of the stochastic gradient at $Z$:
\begin{equation} \label{eqn:surrogate}
	\nabla_{(P, h)} F(x, Z):= P D_{(P, h)} F(x,Z) = \sum\limits_{i=1}^l \frac{F(x + hp^{(i)},Z) - F(x,Z)}{h}p^{(i)}.
\end{equation}
\noindent Two approximations of the classic stochastic gradient of $f$ are performed. The first consists in approximating the gradient using  $l$ random orthogonal directions instead of $d$, namely approximating the gradient with
\begin{equation*}
		\nabla_{P} F(x, Z) := \sum\limits_{i=1}^l \scalT{\nabla F(x, Z)}{p^{(i)}}p^{(i)}=PP^{\intercal} \nabla F(x,Z).
	\end{equation*}
 Note that Assumption~\ref{ass:pk_exp}  implies that
$\mathbb{E}_Z [ \nabla_{P} F(x, Z) ]=\nabla F(x,Z)$.
The second approximation replaces the stochastic directional derivatives with finite differences, see \eqref{eqn:surrogate}. We define the following iterative procedure:
 \begin{algorithm}[H]
 	\caption{S-SZD: Structured Stochastic Zeroth order Descent}\label{alg:szo}
 	\begin{algorithmic}[1]

 		\Require $x_0 \in \mathbb{R}^d$, $(\alpha_k)_{k \in \N}\subset \R_+, (h_k)_{k \in \N}\subset \R_+,l \in \N$ s.t. $1 \leq l \leq d$

 		\For{$k=0,\cdots$}
 		\State build $P_k\in\R^{d \times l}$ satisfying Assumption~\ref{ass:pk_exp}
		
 		\State compute 
 			\begin{equation}\label{eqn:szo}
 				x_{k + 1} = x_k - \alpha_k \nabla_{(P_k, h_k)} F(x_k, Z_k).
 			\end{equation}
	
 		\EndFor		
 	\end{algorithmic}
 \end{algorithm}
\noindent The above iteration has the  same structure of stochastic gradient descent, but the stochastic gradient is replaced  by the surrogate in eq.~\eqref{eqn:surrogate}. Each step  depends on the direction matrix $P_k$ and two free parameters: the step-size $\alpha_k>0$ and the finite difference discretization $h_k>0$. Our main contribution is the theoretical and empirical analysis of this algorithm with respect to these parameters choice. We make the following assumptions.
\begin{assumption}[Summability properties of step-size and discretization]\label{ass:step_disc}
The step-size sequence $\alpha_k>0$ and the discretization sequence $h_k>0$ satisfy
 \begin{equation*}
		\alpha_k \not\in \ell^1, \qquad \alpha_k^2 \in \ell^1 \qquad \text{and} \qquad \alpha_k h_k \in \ell^1.
	\end{equation*}
\end{assumption}
\noindent The previous hypothesis is required in order to guarantee convergence of the iterates of the algorithm to a solution. Note that a similar assumption is used in \cite{kozak2021zeroth}.	An example of choices of step-sizes and discretization parameters that satisfy the assumption above are given by $\alpha_k = \alpha(k + 1)^{-r}$ with $\alpha>0, \ r\in (1/2, 1]$; and $h_k = h(k + 1)^{-s}$ with $h>0, \ s \geq 1/2$.

\subsection{Related work} \label{sec:related_work}
Several zeroth-order methods are available for deterministic and stochastic optimization. Typical  algorithms use directional vectors to perform a direct search, see \cite{comp_search,sds,stp,prob_ds,prob_ds_sketched}, or mimic first order strategies by approximating the gradient with finite differences, see \cite{sto_md,nest,kozak2021zeroth,Kozk2021ASS,pmlr-v162-gasnikov22a,zoro,ZO-BCD,stars,stars_acc,cai2022one}.  \\	
{\bf Direct Search.}
Direct Search (DS) methods build a set of directions $G_k$  (called polling set) at every step $k\in\N$, and  evaluates the objective function at the points $x_k + \alpha_k g$ with $g \in G_k$ and $\alpha_k>0$. If one of the trial points satisfies a sufficient decrease condition, this point becomes the new iterate $x_{k+1}$ and the step-size parameter $\alpha_k$ is (possibly) increased. Otherwise, the current point $x_k$ does not change ($x_{k+1}=x_k$) and the step-size $\alpha_k$ is diminished. These methods can be split in two categories: deterministic and probabilistic. In deterministic DS, the cardinality of the set $G_k$ is at least $d + 1$ and the bound on the number of required function evaluations  such that $\nor{\nabla f(x)} \leq \epsilon$ (with $\epsilon$ tolerance)  is  $O(d^2 \epsilon^{-2})$~\cite{3f5c9a64536d47be8f46be0b369d8c51,Dodangeh2016}. %
Probabilistic direct search has been introduced to improve the dependence on $d$. Polling sets can be   randomly generated and while they  do not form a positive spanning set, they can  still provide a good approximation of each vector in $\R^d$,  \cite{prob_ds,prob_ds_sketched}.  
For example, results in  \cite{prob_ds_sketched} show better complexity bounds using a method based on   a subspace of $\R^d$ and a polling set within this  subspace.

\noindent Deterministic and probabilistic DS methods select only one direction at each step. Every previous function evaluation  on non-descending directions is discarded and different polling sets may be generated  to perform a single step. For instance, this is the case for large  step-sizes and when many reductions have to be performed to generate a descending direction.  A core difference between DS and Algorithm~\ref{alg:szo} is that our method uses all the $l$ sampled  directions to perform an update which is the key feature to provide better performances. \\
{\bf Finite differences based.}
Our approach belongs to the category of finite differences based methods. This class of algorithms was introduced in \cite{KieWol52} (see also \cite {spall}) and has been investigated both for deterministic and stochastic optimization, under various assumptions on the objective function, ranging from smooth to nonsmooth optimization. The vast majority of methods  use a single direction vector to approximate a (stochastic) gradient, and  accelerated approaches are available (see e.g. \cite{nest,stars_acc}). In~\cite{nest,stars,sto_md,pmlr-v162-gasnikov22a} the direction vector is randomly sampled  from a Gaussian  or a spherical distribution. In \cite{zoro,ZO-BCD}, the direction is sampled according to a different distribution, and with a slightly different noise model. The advantages of using multiple directions have been empirically shown on reinforcement learning problems \cite{10.5555/3326943.3327109,Salimans2017EvolutionSA}, and have been theoretically investigated in \cite{zoro,ZO-BCD,sto_md,duchi_smoothing}, where the directions are independently sampled and are not structured.  A key feature of our approach is the use of structured multiple directions, allowing for a better exploration of the space.

\noindent A similar approximation is considered in \cite[Section 2.2]{berahas2022theoretical}. However, the authors analyzed only the quality of the approximation in the case $l = d$ with a different noise model. In \cite{grapiglia2022worst}, a finite difference approximation of the gradient with multiple coordinate directions is proposed but the algorithm is different and it was analyzed only in deterministic setting. In \cite{ozd} a structured finite-difference method for non-smooth optimization is proposed but it has been analyzed only in deterministic setting. In \cite{stief_zeroth} the authors provide a structured zeroth-order method in which structured directions are sampled from the Stiefel manifold \cite{chikuse2012statistics}. However, the analysis is provided only for Lojasiewicz functions \cite{lojasiewicz1963topological} and only the deterministic setting is considered. Moreover, our analysis is performed for a more general class of directions. Our method can be seen as a generalization to the stochastic setting of the approach proposed in  \cite{kozak2021zeroth}. Moreover, if for all $k \in \N$, we set $h_k = h$ and $h$  approaches $0$, iteration ~\eqref{eqn:szo} can be written as 
\begin{equation*}
	x_{k + 1} = x_k - \alpha_k P_k P_k^\intercal \nabla F(x_k, z_k),
\end{equation*}
which is the extension of the algorithm  in~\cite{Kozk2021ASS} to the stochastic setting.	
As already noted our results seem to be the first finite-difference approach to study convergence rates for the stochastic zeroth-order minimization under the \eqref{eq:ABG} condition.\\
{\bf Other methods.}
Different other strategies were proposed in the literature. Model-based methods iteratively build models that approximate the problems (e.g., using polynomials) around the current iterate and choose the next iterate according to the approximated problems. To make model-based methods globally convergent, a globalization strategy (e.g. line-search) or a trust-region are used. For instance, Powell methods \cite{Powell1994,Powell2002,Powell2006,powell2009bobyqa} iteratively explore the parameter space by combining directional searches with line searches, aiming to minimize the objective function using only function evaluations. Trust region methods \cite{Cartis2023,blanchet2019convergence,dzahini2022stochastic,ha2024iteration}, build and adapt an exploration region based on the local behavior of the target function. Other approaches, like \cite{srinivas2009gaussian,adabkb,frazier2018tutorial}, consist in approximating the objective function with a statistical model and using the approximation to select the next iterate (Bayesian optimization). However, these algorithms consider functions defined only on a bounded subset of $\R^d$, and the computational cost scales exponentially in $d$. Recent methods \cite{adabkb,shekhar2018gaussian,Salgia2021ADB} scale better in the number of iterations, but the exponential dependency on the dimension is still present in the cumulative regret bounds. A further class of algorithms is called meta-heuristics \cite{cma,pso,Salimans2017EvolutionSA}. 
Recently there has been an increasing interest in the theoretical study of such methods (see for instance \cite{totzeck2022trends,cbo}), but   convergence rates in terms of function evaluations are not available yet.
\section{Main Results}\label{sec:analysis}
 In the following sections,  we study several convergence properties of our algorithm under different assumptions. We establish the convergence of the iterates and derive convergence rates when $f$ is a convex function. Subsequently, we derive convergence rates for $f$ satisfying the Polyak-Łojasiewicz inequality (see \cite{bolte}), without requiring convexity. %

\subsection{Convex case}
 In this section we assume the following.
 \begin{assumption}[Convexity]\label{ass:fun_conv}
	For every $z\in\mathcal{Z}$, the map $x\mapsto F(x,z)$ is convex.
\end{assumption}
\noindent Notice that Assumption~\ref{ass:fun_conv} implies that $f$ is convex. Next, we state the main theorem on convergence of the iterates and on rates for the functional values. Then, we show explicit rates for specific choices of $\alpha_k$ and $h_k$.

\begin{theorem}[Convergence and rates]\label{thm:creg}
    Under Assumptions \ref{ass:lam_smooth}, \ref{ass:pk_exp}, \ref{ass:step_disc} and \ref{ass:fun_conv}, suppose also that $\alpha_k\leq \bar{\alpha}$ with $\bar{\alpha} d (A+\lambda B)/ l <1$. Let $(x_k)$ be the sequence generated by Algorithm \ref{alg:szo}. 
    Then, %
    we have 
    \begin{equation*}
        \lim\limits_{k} f_k = \min f
    \end{equation*}
	and there exists a random variable $x^*$ supported in $\argmin f$ such that $x_k \rightarrow x^*$ a.s.\\
 Moreover, let $(\hat{x}_k)$ be the sequence of the averaged iterates, namely 
 $$\hat{x}_k:=\frac{1}{\sum_{j=0}^k \alpha_j} \ \sum_{j=0}^k \alpha_j x_j.$$ 
 Then, for every $\delta \in \left(0, 1-d (A+\lambda B) \bar{\alpha}/l\right]$ and for every $\bar{x} \in \argmin f$, we have that
\begin{align}\label{eq:prerate0}
&\E[f(\hat{x}_k) - f^*]  \leq \frac{1}{2 \delta \sum\limits_{j=0}^k\alpha_j} \left[S_k+  \sum\limits_{j=0}^k \rho_j \left(\sqrt{S_{j}} + \sum\limits_{i=0}^{j}\rho_i\right)\right],
\end{align}
where for every $k \in \mathbb{N}$
 \begin{equation}
S_k := \nor{x_0 - \bar{x}}^2 + \sum\limits_{j=0}^{k} W_j
\end{equation}
 and,
 \begin{equation}
		W_k :=  \frac{ d}{l} \Bigg( \frac{ \lambda^2 d^2}{2l}h_k^2 + 2G \Bigg) \alpha_k^2, \quad \rho_k :=  \frac{\lambda d \sqrt{d}}{l} \alpha_k h_k.
  \end{equation}
\end{theorem}
\noindent The proof of the previous theorem can be found in Appendix \ref{app:proofs_abg_cond} - see Proposition \ref{prop_convx} and Theorem \ref{th_convex}. 
\begin{remark}
    Note that, if $A=B=0$, meaning that the expectation of the squared norm of the stochastic gradients is bounded, we do not need to assume any upper-bound on $\bar{\alpha}$. As we will see in Corollary~\ref{cor:conv}, this will allow to derive bounds with a better dependence on the dimension.
\end{remark}
\begin{remark}
    Under Assumption \ref{ass:step_disc}, the sequences $W_k$ and $\rho_k$ are both summable and so $S_k$ is bounded. Then the rate in \eqref{eq:prerate0} can be written as
    \begin{align*}
&\E[f(\hat{x}_k) - f^*]  \leq C/\sum\limits_{j=0}^k\alpha_j 
\end{align*}
for some constant $C>0$. Again from Assumption \ref{ass:step_disc}, we know that $\sum_{j=0}^k \alpha_j$ diverges for $k\to + \infty$ so that, from the inequality above, $f(\hat{x}_k)$ converges to $f^*$ in expectation.
\end{remark}
\noindent To clarify the previous result we state a corollary in which we make  explicit the dependence of the right-hand-side both on the dimensions $d$ and on the number of directions $l$.
\begin{corollary}\label{cor:conv}
    Let Assumptions \ref{ass:lam_smooth},  \ref{ass:pk_exp} and \ref{ass:fun_conv} hold. Set $\alpha_k = \bar{\alpha} a_k$, where  $\bar{\alpha}>0$ and $0<a_k \leq 1$. If $A+\lambda B >0$, set $\bar{\alpha}= l \eta / [d(A + \lambda B)]$
    for some $0<\eta <1$. Set also $h_k =\bar{h} \xi_k$ with $\bar{h}>0$ and $0<\xi_k\leq 1$. Moreover, assume $a_k^2 \in \ell^1$ and $a_k \xi_k \in \ell^1$. Let $(\hat{x}_k)$ be the sequence of the averaged iterates generated by Algorithm \ref{alg:szo}. Then we have, for some constant $C>0$,
    \begin{equation}
        \begin{aligned}\label{ineqqq}
        \E[f(\hat{x}_k) - f^*]  \leq \frac{C}{\bar{\alpha} {\sum_{j=0}^k a_j}} \Bigg[& \|x_0 - \bar{x}\|^2 + \frac{\lambda^2 d^3 \bar{\alpha}^2 \bar{h}^2}{l^2} +  \frac{G d\bar{\alpha}^2}{l}\\
        &+ \|x_0 -\bar{x}\| \frac{\lambda d\sqrt{d} \bar{\alpha} \bar{h} }{l} + \frac{\lambda \sqrt{G} d^2 \bar{\alpha}^2 \bar{h}}{l \sqrt{l}}\Bigg].            
        \end{aligned}
    \end{equation}
    \ \\
    Now consider the case $A+\lambda B >0$ and so $\bar{\alpha}= l \eta / [d(A + \lambda B)]$ for $0<\eta<1$. Then, showing only the dependence with respect to $d, l$ and $\bar{h}$, for some constant $C_1>0$, we have
    \begin{equation}\label{ineq1}
        \begin{aligned}
        \E[f(\hat{x}_k) - f^*]  &\leq \frac{C_1}{ {\sum_{j=0}^k a_j}} \ \Bigg[ \frac{d}{l} + \frac{d^2 \bar{h}^2}{l} +  1+\frac{ \bar{h} d \sqrt{d}}{l} + \frac{d \bar{h}}{\sqrt{l} }\Bigg].            
        \end{aligned}
    \end{equation}
    In particular, choosing also $\bar{h}$ as a suitable function of $d$ and $l$, there exists a constant $D_1>0$ such that
    \begin{equation*}
        \begin{aligned}
        \E[f(\hat{x}_k) - f^*]  &\leq  \left(\frac{d}{l}\right) \frac{D_1} {\sum_{j=0}^k a_j} .
        \end{aligned}
\end{equation*}
Now consider the case $A=B=0$ and choose $\bar{\alpha}=\sqrt{l/d}$. Then, showing only the dependence with respect to $d, l$ and $\bar{h}$, for some constant $C_2>0$, we get
\begin{equation}\label{ineq2}
        \begin{aligned}
        \E[f(\hat{x}_k) - f^*]  &\leq \frac{C_2}{ {\sum_{j=0}^k a_j}} \Bigg[ \frac{\sqrt{d}}{\sqrt{l}} + \frac{d^2\sqrt{d} \bar{h}^2}{l\sqrt{l}} + \frac{d\sqrt{d} \bar{h}}{l}\Bigg].      
        \end{aligned}
    \end{equation}
    In particular, choosing also $\bar{h}$ as a suitable function of $d$ and $l$, there exists a constant $D_2>0$ such that
    \begin{equation*}
        \begin{aligned}
        \E[f(\hat{x}_k) - f^*]  &\leq \left(\sqrt{\frac{d}{l}}\right) \frac{D_2}{ {\sum_{j=0}^k a_j}} .  
        \end{aligned}
    \end{equation*}
\end{corollary}

\noindent The proof of the previous corollary can be found in Appendix \ref{app:proof_cor_conv}.%

\begin{corollary}
Under Assumptions \ref{ass:lam_smooth}, \ref{ass:pk_exp} and \ref{ass:fun_conv}, set $\alpha_k = \bar{\alpha} (k+1)^{-r}$ with $\bar{\alpha}{d(A + \lambda B)}/l<1$, $r \in \left(1/2, 1\right]$; and set $h_k = ({h}/{d\sqrt{d}})(k+1)^{-s}$ with ${h} >0$ and $s \in [1/2, +\infty)$. Let $(\hat{x}_k)$ be the sequence of the averaged iterates generated by Algorithm \ref{alg:szo}. Then we have that, for every $k \in \mathbb{N}$, there exist positive constants $D_1$ and $D_2$ such that
\begin{equation*}
    \begin{aligned}
    \E[f(\hat{x}_k) - f^*] \leq \begin{cases}  
    \dfrac{d}{l}\dfrac{{D_1}} {(k+1)^{1-r}}   & \text{if } A+\lambda B >0 \\
    \ &\\
    \sqrt{\dfrac{d}{l}} \dfrac{D_2} {(k+1)^{1-r}}  & \text{if } A= B = 0.
    \end{cases}
    \end{aligned}
\end{equation*}
Moreover, the number of function evaluations required to obtain an error $\varepsilon \in (0, 1)$ is
\begin{equation*}
    \mathcal{O}\left( \frac{1}{l^{\frac{r}{1-r}}} \left(\frac{d}{ \varepsilon} \right)^{\frac{1}{1 - r}} \right).
\end{equation*}
\end{corollary}
\paragraph*{Discussion} 
The bound in Theorem \ref{thm:creg} depends on three kinds of terms: the initialization $\nor{x_0 - x^*}^2$; the errors generated by the finite differences approximation, involving the discretization $h_k$; and a term representing the stochastic noise, involving the bound on the variance $G$. Since the convergence rate is inversely proportional to the sum of the step sizes, we would like to take the step sizes as large as possible. On the other hand, from Assumption \ref{ass:step_disc}, the sequence of the step sizes $\alpha_k$ needs to converge to zero for $k\to +\infty$, to make the stochastic error vanish. A balancing choice for this trade-off is given in Corollary \ref{cor:conv}. Note that our result considers a more general setting than state-of-the-art approaches. Indeed, state-of-the-art methods are generally analyzed under  stronger versions of Assumption~\eqref{ass:lam_smooth}(i) (i.e. where $A = 0$ and $B = 1$ or even $A=B=0$). Our result is  similar to the one in \cite{kozak2021zeroth} but slightly worse due to the stochastic setting. In particular, for $r$ approaching to $1/2$, the rate approaches $1/\sqrt{k}$, which is the rate for SGD \cite{garrigos2023handbook}. Note that the choice of $s$ has not a direct impact on the rate since it affects only terms that decrease faster than $1/k^{1-r}$ - see Appendix \ref{app:proof_cor_conv}. The setting of \cite[Assumption B]{sto_md} corresponds to the choice $A=B=0$, and in this case, as in \cite{sto_md}, we get a dependence from the dimension of $\sqrt{d/l}$ instead of $d/l$. In the general case as well, for $k > d/l$, we have that $\alpha_k = \sqrt{l/d} \lambda^{-1} (1+k)^{-r}$ satisfies Assumption~\ref{ass:step_disc} and therefore also in this case we can derive the same dependence on the dimension. 

\noindent Note that the rate, given in terms of function evaluations, improves when $l$ increases. This shows the importance of taking more than one structured direction at each iteration (see Section~\ref{sec:experiements} for the empirical counterpart of this statement). This means that, if we fix a computational budget, it is more convenient to make less iterations with a bigger $l$ (and in fact $l=d$) rather than more iterations with $l=1$. 

\noindent With $l=1$, we get a similar complexity bound to Probabilistic Direct Search~\cite{prob_ds,prob_ds_sketched}. Note that our complexity bound is better than deterministic direct search bound~\cite{3f5c9a64536d47be8f46be0b369d8c51,Dodangeh2016}. However, an exact comparison of our results with Direct Search methods is not straightforward since the considered quantities are different. Moreover, note that our analysis neither relies on smoothing properties of finite differences e.g. \cite{nest,flaxman2005online,stars} nor on a specific distribution \cite{nest,duchi_smoothing,flaxman2005online} from which directions are sampled. Therefore the results hold for a more general class of directions (e.g. our analysis works also using coordinate directions). Moreover, note that our analysis also provides the convergence of iterates that was not considered in previous works \cite{nest,duchi_smoothing,flaxman2005online} in stochastic setting. Finally, we remark that the same bounds can be stated for the best iterate, instead of the averaged one.
 
\subsection{Non-convex setting}

We now consider the non-convex setting, where we suppose only that the objective  $f$ satisfies the Polyak-Łojasiewicz condition; see \cite{bolte}.

\begin{assumption}[Polyak-{\L}ojasiewicz condition]\label{ass:pl}
	The function $f$ is $\gamma$-PL; namely, it is differentiable and there exists $\gamma>0$ such that
	\begin{equation*}
		(\forall x \in \R^d) \qquad \nor{\nabla f(x)}^2 \geq \gamma (f(x) - f^*).
	\end{equation*}
\end{assumption}
\noindent Recall that, for instance, if $f$ is $\mu$ strongly convex, then Assumption \ref{ass:pl} holds with $\gamma=2\mu$. 

\begin{remark}
    Suppose that Assumption \ref{ass:pl} holds. Then, if Assumption~\eqref{ass:lam_smooth}(i) holds with $A,\ B$ and $G$, it holds also with $\tilde{A}=0,\  \tilde{B}=B+A/\gamma$ and $\tilde{G}=G$.
\end{remark}

\begin{theorem}[Convergence Rates for PL functions]\label{thm:abg_pl_rates}
	Under Assumptions \ref{ass:lam_smooth}, \ref{ass:pk_exp} and \ref{ass:pl}, suppose in addition that $\alpha_k\leq\bar{\alpha}$ with  $\bar{\alpha} d\lambda B /l \leq 1$. Let $(x_k)$ be the random sequence generated by Algorithm~\ref{alg:szo} and $f_k:=f(x_k)$. Then, for every $k\in \N$,
 \begin{equation}\label{eqn:nonc_b1}
        \begin{aligned}        
           \E \left[f_{k+1}\right] - f^* \leq & \left[1-\alpha_k \Big( \frac{\gamma}{2}-\frac{d\lambda\left(A+\gamma B\right)}{l}\bar{\alpha} \Big) \right] \left[\E \left[f_k\right] - f^*\right] + C_k,
        \end{aligned}
	\end{equation}
where
 \begin{equation*}
        C_k := \frac{\lambda  d}{l} \alpha_k^2 G+ \alpha_kh_k^2\Big(\frac{\lambda d }{2\sqrt{l}}\Big)^2\Big(\frac{1}{2}+\frac{\lambda d}{l}\alpha_k\Big).
    \end{equation*}
\end{theorem}
\noindent The proof of the previous theorem is given in Appendix \ref{app:nonconv1}.

\noindent Note that the bound in the previous theorem is useful only if the coefficient of $\E \left[f_k\right] -f^*$ is strictly less than 1. To ensure this property, it is enough to assume a more stringent condition on $\bar{\alpha}$, namely  $2 \bar{\alpha} d\lambda (A+\gamma B) /(\gamma l) \leq 1$. 

\noindent The bound in equation \eqref{eqn:nonc_b1} is  composed by three parts: a term depending from the previous iterate, a term representing the approximation error (involving $h_k$) and a term representing the error generated by the stochastic information (involving $G$). 
In the following corollary, we explicitly show convergence rates by choosing $\alpha_k$ and $h_k$.
\begin{corollary}\label{cor:nonconv}
Under the same assumptions of Theorem \ref{thm:abg_pl_rates}, let $\bar{\alpha}>0$ and choose $\alpha_k = \bar{\alpha} (k + 1)^{-\theta}$ and $h_k = h (k + 1)^{-\theta/2}$ with $1/2 < \theta < 1$, $h > 0$.
If $A+\gamma B >0$, set $2\bar{\alpha}\lambda d(A + \gamma B)= l \eta$ for some $0\leq\eta <1$.
 Then, showing only the dependence with respect to $d, l, \bar{\alpha}$ and $h$, we have that, for every $k \in \mathbb{N}$,
\begin{equation*}
\mathbb{E}[f_k - f^*] \leq
    \dfrac{2 \lambda d}{\gamma l (1-\eta)} \Big[G\bar{\alpha}+\dfrac{h^2\lambda d}{8}+\dfrac{h^2\lambda^2 d^2 \bar{\alpha}}{4l} \Big](k + 1)^{-\theta}+o((k + 1)^{-\theta}) 
\end{equation*}
\end{corollary}

\noindent The proof of this corollary is in Appendix~\ref{app:nonconv2}.%
\paragraph*{Discussion}  
Theorem~\ref{thm:abg_pl_rates}, in the PL setting, proves convergence in expectation to the global minimum of the function. To the best of our knowledge, Theorem~\ref{thm:abg_pl_rates} is the first convergence result for a structured finite-difference method in a stochastic setting for non-convex functions under the \eqref{eq:ABG} condition. In the more classical scenario where $A=B=0$, the upper bound on the stepsize $\bar{\alpha}$ is arbitrary (and in this case $\eta=0$). Looking at the statement of Theorem~\ref{thm:abg_pl_rates}, convergence cannot be guaranteed for a constant step-size. This is due to the presence of non-vanishing stochastic and discretization errors. For a constant step-size choice, we can however observe a linear convergence towards the minimum plus an additive error. As we  notice looking at the constants, increasing the number of directions improves the worst case bound in the non-convex PL case. Indeed, when we increase $l$, the error in the constants becomes smaller. Convergence is then derived by suitably choosing decreasing step-sizes and discretization parameters. In the corollary, we derive a convergence rate arbitrarily close to $1/k$, which is the same of the SGD in the strongly convex case. As for first order methods, the convergence rate depends on  $\gamma$. 
The rates derived in Corollary~\ref{cor:nonconv} show that, for suitable choices of $h$ (and $\bar{\alpha}$ in the case $A+\gamma B=0$), the constants in the leading term do not depend on the dimension. Such a dependence remains in the $o$-term. For instance, if $A+\gamma B >0$
choosing $\eta=1/2$ and $h=\sqrt{\ell}/d$, the constant in front of the leading term is
\[
\frac{4G + \lambda^2 (A+\gamma B+1)}{ 4\gamma(A+\gamma B).} 
\]
Comparing the result obtained in Theorem~\ref{thm:abg_pl_rates} and Corollary~\ref{cor:nonconv} with the analysis of PL case in~\cite{kozak2021zeroth}, we  observe that we get a similar bound. However, since we are in stochastic case and we cannot take $\alpha_k$ constant or lower-bounded.

\subsection{Expanded discussion}\label{app:exp_discussion}
In this section, we compare S-SZD~\ref{alg:szo} with stochastic finite-difference methods proposed in previous works.  We compare convergence rates obtained for solving problem \eqref{eqn:problem} considering different settings. 
In \cite{nest}, the authors propose a zeroth-order algorithm for stochastic non-smooth optimization. They assume $L$-Lipschitz continuity of the function and that for every $x \in \mathbb{R}^d$, $\|x - x^*\| \leq \bar{R}$ for some $\bar{R} < + \infty$. Fixing a maximum number of iterations $K$ and choosing $\alpha_k \leq \bar{R} / ((d + 4)\sqrt{(K + 1)}L)$ and $h_k = \epsilon / (2 L \sqrt{d})$ with $\epsilon \in (0, 1)$, they obtain a convergence rate of $\mathcal{O} ({d}/\sqrt{k})$. With S-SZD, we obtain $\mathcal{O}(({d}/{l}){1/k^{1 - r}})$ with $1/2 < r < 1$.  Note that, taking $r = 1/2 + \delta$ with $\delta > 0$ and $l=1$, our rate is arbitrary near to the rate in \cite{nest}. For $l > 1$, we get a better dependence on dimension showing that taking more than one direction provides better performances. Moreover, we do not need to assume that $\|x - x^*\| \leq R$. Note also that the setting we consider is different: we do not need to assume Lipschitz continuity of the function but we assume $\lambda$-smoothness (Assumption \ref{ass:lam_smooth}). 
In \cite{sto_md}, the authors propose a zeroth-order algorithm for stochastic smooth optimization which approximates the gradient with a set of unstructured directions. The convergence rate in convex case is $\mathcal{O} \Big((1 + d/l)^{1/2}/k^{1/2} \Big)$.
However, they obtain this rate by choosing the following parameters $\alpha_k = (\alpha R)/(2G \max\{\sqrt{d/l}, 1\} \sqrt{k})$ and $h_k = (h G)/(\lambda d^{3/2} k)$. Note that the dependence on the iteration $k$ is slightly better than ours (depending on the choice of $r$). However, they consider a more restrictive setting. Specifically, they assume that for every $x$, $\nor{x - x^*}^2 \leq (1/2) R^2$ for some $R < + \infty$ - see \cite[Assumption A]{sto_md} while we do not make this assumption. Moreover, they assume that $\E[\nor{\nabla F(x, z)}^2]$ is uniformly bounded by a constant (see \cite[Assumption B]{sto_md}) which corresponds to the choice $A = B = 0$. In such a scenario, we obtain the same dependence on the dimension. 
In \cite{ghadimi_lam}, the authors provide a finite-difference method for stochastic optimization that in the convex setting, has a convergence rate of $\mathcal{O}(\sqrt{d / (k + 1)})$.  This result is obtained fixing a priori the maximum number of iteration $K > 0$ and choosing the stepsize and the discretization parameter as 
\begin{equation*}
    \alpha_k = \frac{1}{\sqrt{d + 4}} \min \left\{\frac{1}{4L(\sqrt{d + 4})}, \frac{\tilde{D}}{G \sqrt{K + 1}} \right\} \quad \text{and} \quad h_k \leq \frac{\|x_0 - x^* \|}{\sqrt{d + 4}},
\end{equation*}
for some $\tilde{D} > 0$ - see \cite[Corollary 3.3]{ghadimi_lam}. Note that, this implies that the dependence on the dimension in the step-size is $1/\sqrt{d}$ when $K > \mathcal{O}(d)$. Such results hold under \cite[Assumption A1]{ghadimi_lam}, which  corresponds to the choice $A = 0$ and $B = 1$. Note that, in such a setting, choosing $\alpha_k = \sqrt{l/d} \lambda^{-1} (k + 1)^{-r}$ for $k > d/l$, we derive a dependence on the dimension of the order $\mathcal{O}(\sqrt{d/l})$, with no need to fix a priori the maximum number of iterations.  %
In Table~\ref{tab:expanded_discussion} we summarize the convergence rates of S-SZD and \cite{nest,sto_md,ghadimi_lam} along with the considered  choice of parameters. Note that the convergence results obtained in \cite{nest,ghadimi_lam} are obtained fixing a maximum number of iterations $K > 0$. We underline that the non-convex case is not analyzed in \cite{nest} and \cite{sto_md}, and the non-convex Polyak-Łojasiewicz case is not considered in \cite{ghadimi_lam}. Additionally, our setting is more general than those considered in \cite{sto_md,ghadimi_lam}, which correspond to the choices $A = B = 0$ and $A = 0, \, B = 1$, respectively. Note that for the deterministic setting, the result obtained in \cite{kozak2021zeroth} hold since our algorithm is its direct generalization to the stochastic setting.

 \begin{table}[ht]
  	\caption{Comparison of finite-difference based algorithms in stochastic convex setting. The rate is on the expected objective function value. Here $K > 0$ is the maximum number of iteration}
  	\label{tab:expanded_discussion}
  	\centering
  	\begin{tabular}{lllll}
  		\toprule
  		Algorithm & Rate &  $\alpha_k$ & $h_k$\\
  		\midrule
  		Nesterov et al. \cite{nest} & $\mathcal{O} \left(\frac{d}{\sqrt{K}}\right)$ & $\frac{\bar{R}}{(d + 4)\sqrt{K}L}$  & $\frac{\epsilon}{2 L \sqrt{d}}$ \\
  		Duchi et al. \cite{sto_md} & $\mathcal{O}\left(\sqrt{\frac{d}{l}} \frac{1}{\sqrt{k + 1}}\right)$ &  $\frac{\alpha R}{2G \max\{\sqrt{d/l}, 1\} \sqrt{k + 1}}$ & $\frac{hG}{\lambda d^{3/2} (k + 1)}$\\
         Ghadimi \& Lan \cite{ghadimi_lam}  & $\mathcal{O}\left(\sqrt{\frac{d}{K}} \right)$ & $\frac{1}{\sqrt{d + 4}} \min \left\{\frac{1}{4L(\sqrt{d + 4})}, \frac{\tilde{D}}{G \sqrt{K}} \right\}$ & $\frac{\|x_0 - x^* \|}{\sqrt{d + 4}}$\\
  		\textbf{S-SZD} & $ \mathcal{O} \left( \sqrt{\frac{d}{l}} \frac{1}{(k + 1)^{1 - r}} \right)$ &  $ \frac{\bar{\alpha}}{(k + 1)^{r}}$ & $\frac{h}{d\sqrt{d}} \frac{1}{(k + 1)^s}$\\
  		\hline
  	\end{tabular}
  \end{table}

\section{Experiments}\label{sec:experiements}
In this section, we study the empirical performance of S-SZD:   we start by analyzing performances of S-SZD by increasing the number of directions $l$ and then we compare it with different algorithms. We consider  Stochastic Three Points (STP)~\cite{stp}, Probabilistic DS (ProbDS)~\cite{prob_ds}, its reduced space variants (ProbDS-RD)~\cite{prob_ds_sketched}, COBYLA \cite{Powell1994} and other unstructured finite-difference (FD) algorithms. As mentioned in Section \ref{sec:related_work}, finite-difference algorithms emulates first-order methods by approximating the gradient. They differ from S-SZD since they do not enforce any structure on the direction matrices, such as orthogonality. Specifically, we explore direction matrices constructed by sampling one or more random coordinate directions, random Gaussian directions, and random directions uniformly sampled from a sphere. In particular, finite-difference methods with random Gaussian or Spherical directions have been studied in many works - see e.g. \cite{nest,ghadimi_lam,berahas2022theoretical,flaxman2005online,duchi_smoothing}. Note that a convergence analysis for these direct search methods is not available in the stochastic setting \cite[Conclusions]{prob_ds_sketched}, and therefore convergence is not guaranteed. For S-SZD we consider %
structured spherical directions - see Appendix \ref{app:dir_gen} for details. For ProbDS/ProbDS-RD we use independent spherical and orthogonal directions. Sketching matrices for ProbDS-RD are built by using orthogonal vectors - for details see~\cite{prob_ds_sketched}. The number of directions used for polling matrices of ProbDS is $2$, the sketching matrices ProbDS-RD have size $l = d/2$. For finite-difference methods, we consider the algorithms proposed in Flaxman et al. \cite{flaxman2005online}, Berahas et al. \cite{berahas2022theoretical}, Ghadimi \& Lan \cite{ghadimi_lam}, Duchi et al. \cite{duchi_smoothing}, a finite-difference method employing a single random coordinate direction (SCD) and another using $l=d$ coordinate directions (Deterministic FD). Note that the first four algorithms are based on non-structured finite-differences approximations. The latter two algorithms can be viewed as particular cases of S-SZD, employing coordinate directions with a number of directions equal to $1$ and $d$, respectively. We refer to Appendix~\ref{app:experiment_details} for further details and results.
\paragraph*{Increasing the number of directions} In these experiments, we fix a budget of $50000$ function evaluations and we observe how the performances of S-SZD change by increasing $l$. We report mean and standard deviation using $10$ repetitions.
\begin{figure}[H]
	\centering
    \includegraphics[width=\linewidth]{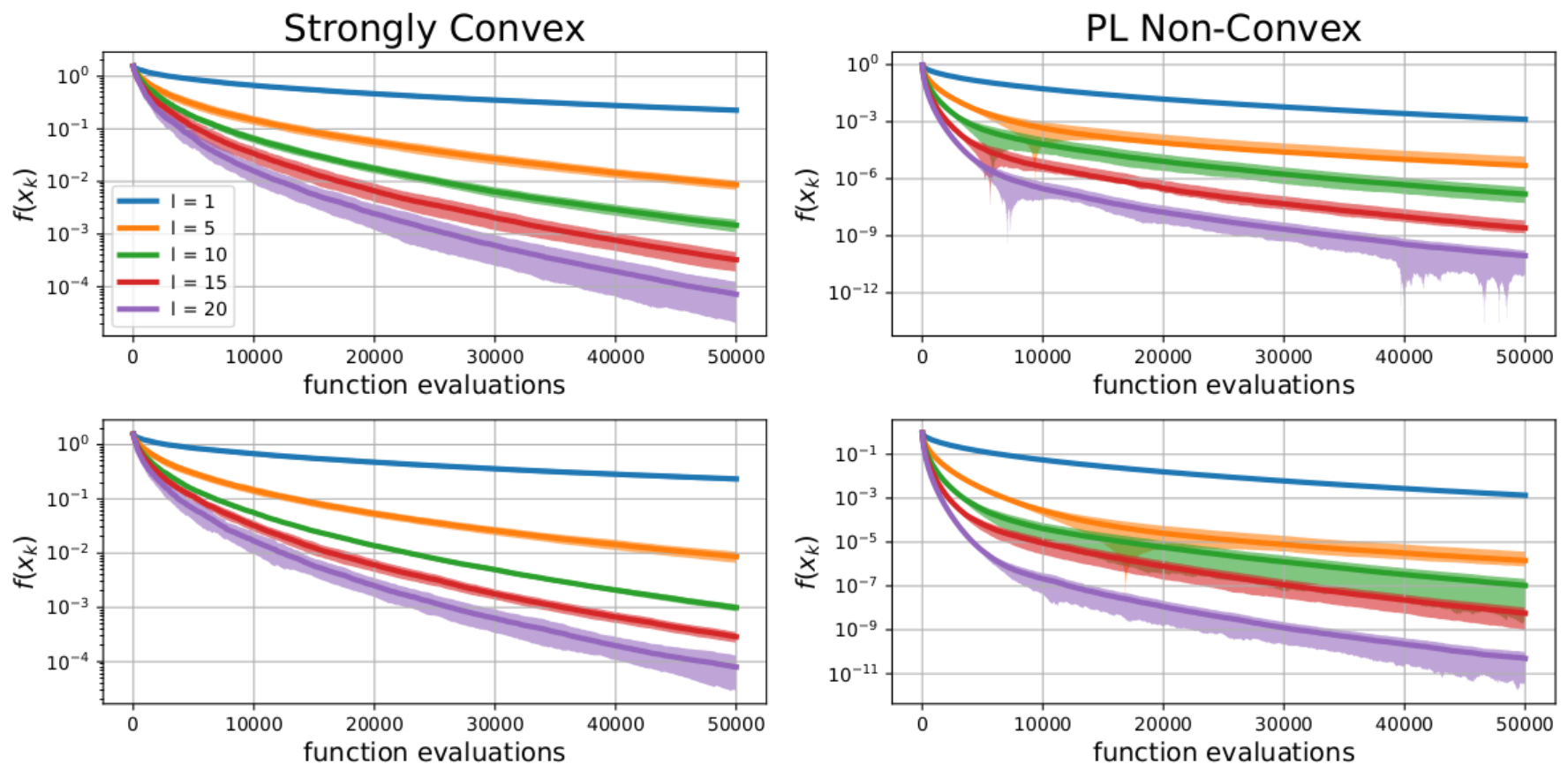}
	\caption{From left to right,  function values per function evaluation in optimizing a strongly convex and a PL non-convex target functions using S-SZD with different numbers of directions. In the first row, direction matrices are built using structured spherical directions. In the second row, coordinate directions are used. Details on the considered functions can be found in Table~\ref{tab:syn_fun}, Appendix~\ref{app:syn_exp_details}}.
	\label{fig:changing_l_1}
\end{figure}
\noindent In Figure~\ref{fig:changing_l_1}, we show the objective function values, which are available in this controlled setting after each function evaluation. When $l > 1$, values are repeated $l$ times indicating that we have to perform at least $l + 1$ function evaluations to update the objective function value. Note that for a budget large enough, increasing the number of directions $l$, provides better results w.r.t. a single direction. 
\paragraph*{Illustrative Experiments} We consider the problem of minimizing functions that satisfy Assumptions \ref{ass:fun_conv} and \ref{ass:pl}. Specifically, we consider the optimization of three functions - details in Appendix~\ref{app:experiment_details} and Table \ref{tab:syn_fun}. We plot the mean and standard deviation using $10$ repetitions.

\begin{figure}[H]
	\centering
    \includegraphics[width=0.9\linewidth]{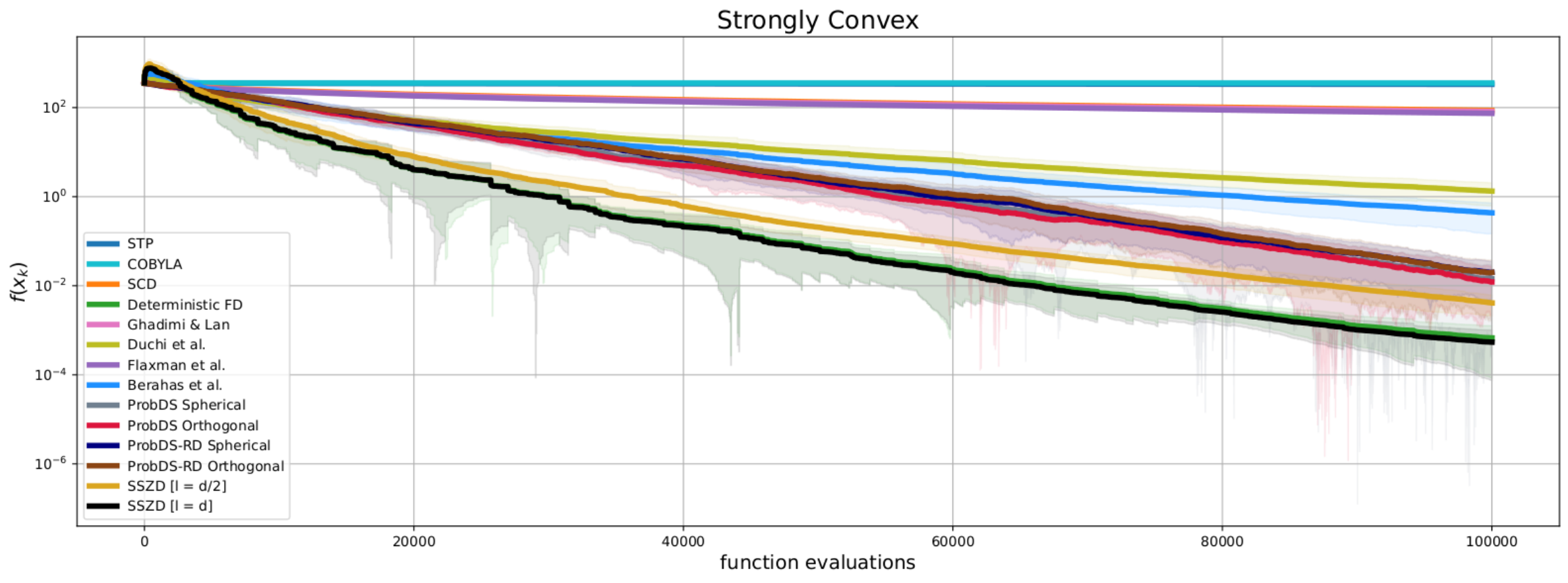}
    \includegraphics[width=0.9\linewidth]{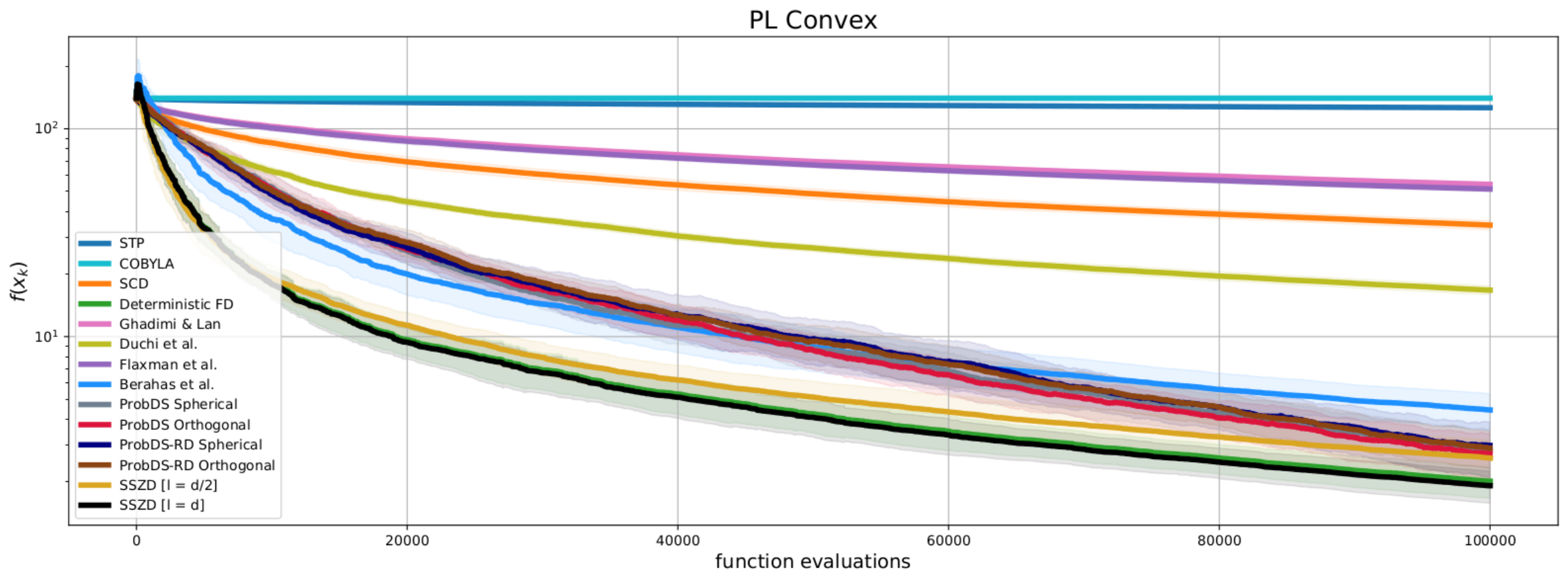}
    \includegraphics[width=0.9\linewidth]{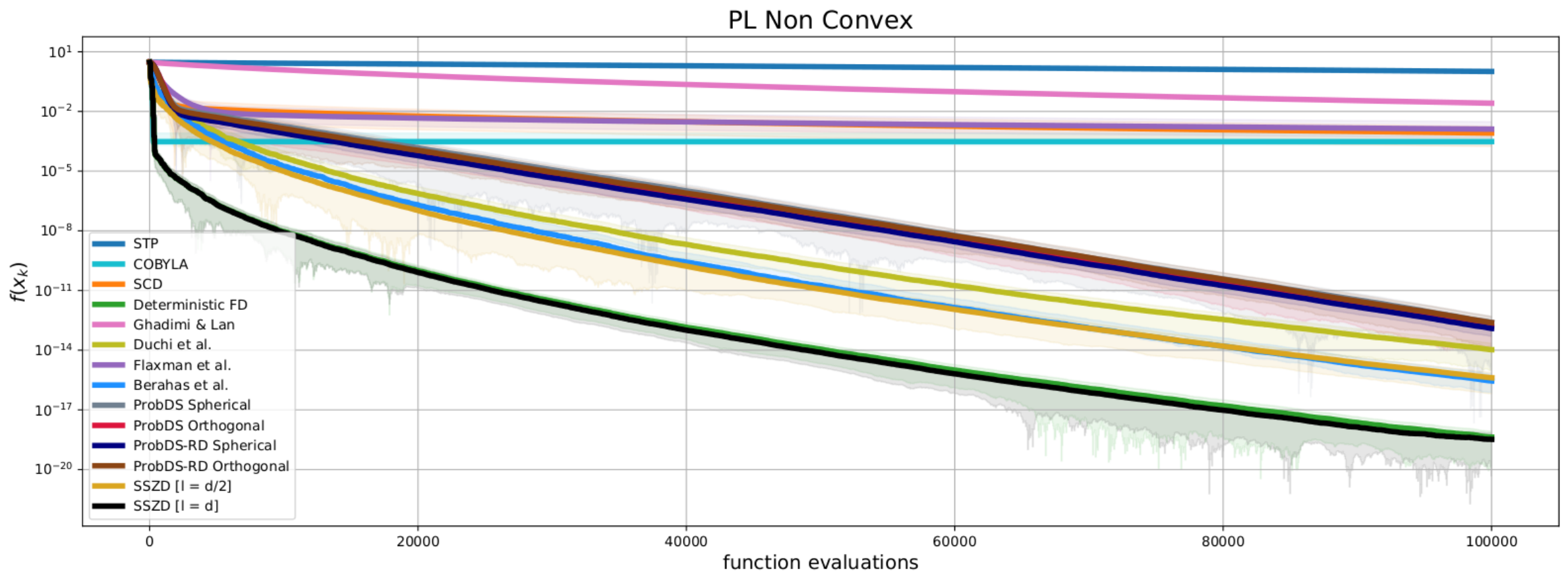}
	\caption{Comparison of SSZD with different algorithms on a strongly convex, a PL convex and a non-convex PL function.}
	\label{fig:synthetic}
\end{figure}
\noindent In Figure \ref{fig:synthetic}, we plot the objective function values with respect to the number of function evaluations. As we can observe S-SZD obtains the highest performances. In particular, we observe that S-SZD provide better performance than the other finite-differences methods. This result can be explained by observing that a gradient approximation built with Gaussian finite differences as in \cite{nest,sto_md} can be arbitrarily bad even when the directions chosen are near the direction of the gradient (this is because Gaussian distribution has infinite support). Note also that structured directions provide a better local exploration of the space than non-structured ones, reducing the probability to generate bad directions (i.e. when all directions chosen are far from the direction of the gradient). This result is also confirmed in terms of gradient accuracy in \cite{berahas2022theoretical} where the authors empirically show that to obtain a gradient accuracy comparable to methods that use orthogonal directions, smoothing methods (Gaussian or spherical) can require significantly more samples. Moreover, notice that S-SZD and other finite-difference methods perform an update every $l + 1$ function evaluations while direct search methods may (or not) need more function evaluations since the step size has to be reduced in order to find a descent direction. %
\paragraph*{Falkon Tuning} In this section, we consider the problem of parameter tuning for the large-scale kernel method Falkon \cite{rudi2017falkon}. Falkon is based on a \Nystrom{} approximation of kernel ridge regression. In these experiments, the goal is to minimize the hold-out cross-validation error with respect to two classes of parameters: the length-scale parameters of the Gaussian kernel $\sigma_1,\cdots,\sigma_d$, in each of the $d$ input dimensions; the regularization parameter $\lambda$. The dimension of the input space are $8$ and $9$ - see additional details in Appendix~\ref{app:flk_exp_details}. We report mean and standard deviation using $5$ repetitions.  
\begin{figure}[H]
	\centering
    \includegraphics[width=0.9\linewidth]{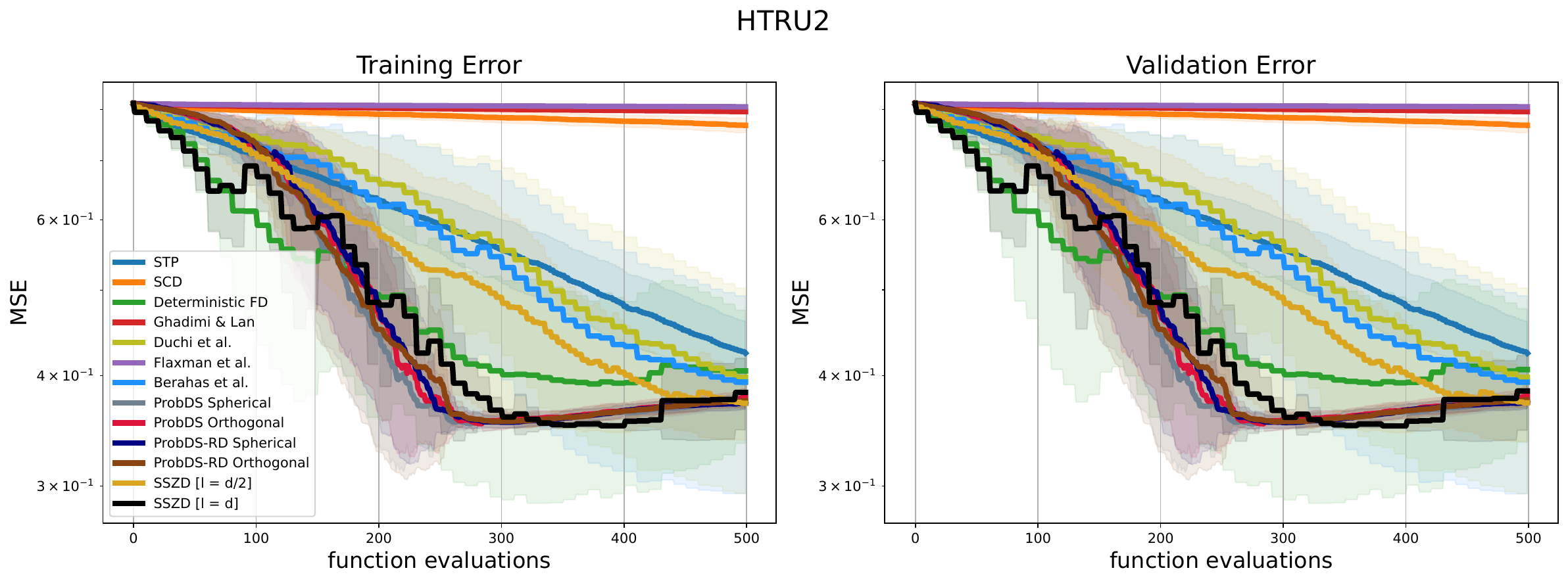}\\
    \includegraphics[width=0.9\linewidth]{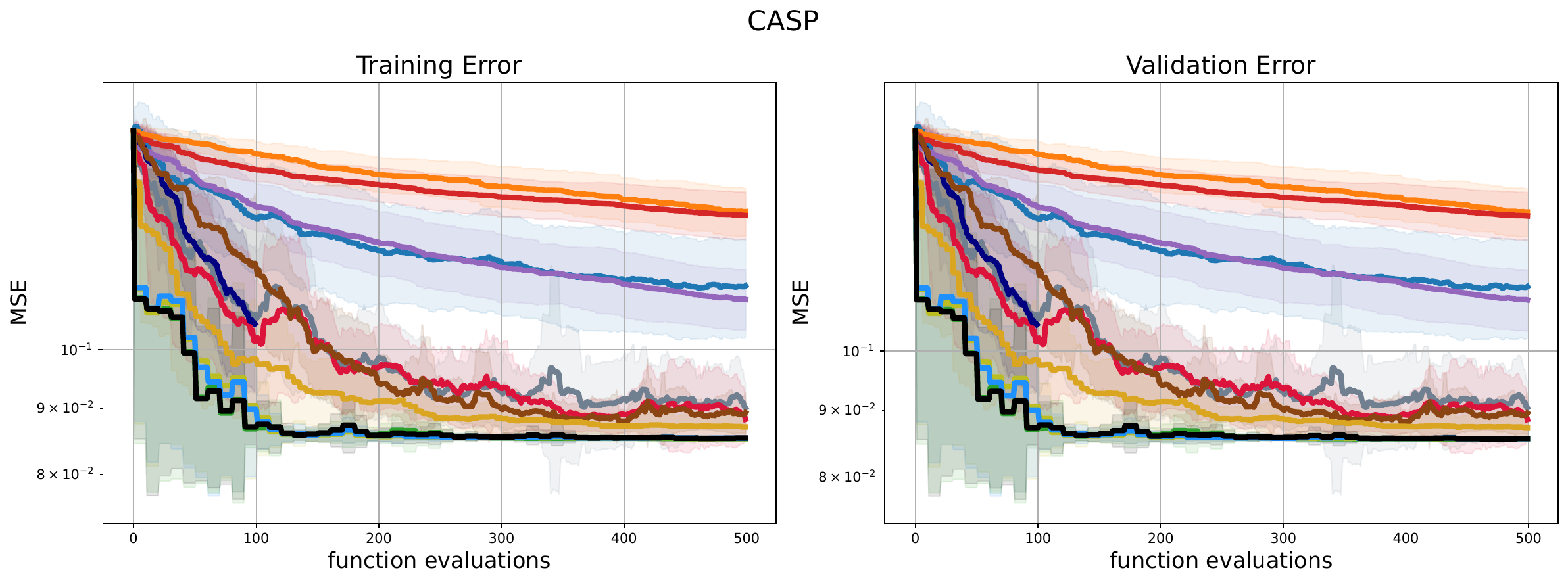}\\
    \includegraphics[width=0.9\linewidth]{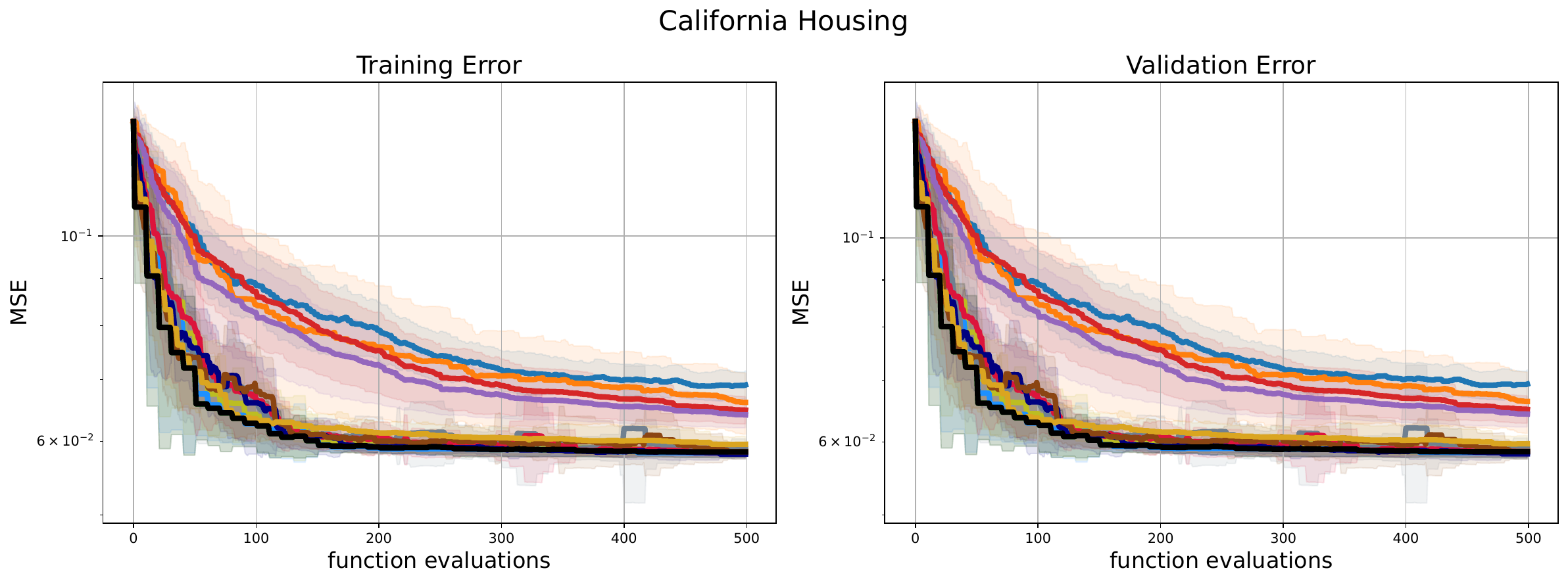}
	\caption{From left to right, training  and validation error obtained by tuning FALKON with different optimizers. From top to bottom, results for HTRU2, CASP and California Housing datasets}
	\label{fig:flk_tuning}
\end{figure}
\noindent From Figure~\ref{fig:flk_tuning}, we observe that Algorithm~\ref{alg:szo} provides the best performances considering the validation error. Despite the iterations of STP are cheaper (it performs $3$ function evaluations per iteration, while S-SZD needs $\ell+1>3$ function evaluations per iteration), STP does not perform well on the validation error. Comparing S-SZD with the other algorithms (STP excluded), we see that it achieves the best performances almost always. The test error obtained with Falkon tuned with S-SZD is similar to the best result of other methods - see also Table~\ref{tab:flk_tuning_te}. To conclude, our algorithm provides better performance than direct search methods in simulated experiments and behaves similarly in parameter tuning experiments. Notice that S-SZD achieves comparable performance when using $l = d / 2$, suggesting that for certain problems, employing all $d$ directions may be unnecessary.
\section{Conclusion}\label{sec:conclusions}
We presented and analyzed S-SZD, a stochastic zeroth-order algorithm for smooth minimization. We derived convergence rates for smooth convex functions and nonconvex functions, under  the Polyak-{\L}ojasiewicz assumption. We empirically compared our algorithm with a variety of zeroth order methods, and the numerical results suggest  that our algorithm outperforms direct search methods when the assumptions needed for convergence are satisfied. In addition, our algorithm works well also on practical parameter tuning experiments. Our work opens a number of possible research directions. In particular, an interesting question could be the development of an adaptive strategy to choose the direction matrices along the iterations. Moreover, another interesting question is to study adaptive choices of the stepsize and the discretization parameter. For instance, including line-search procedure.\\
\textbf{Acknowledgments \& Fundings.}
L. R. acknowledges the financial support of the European Commission (Horizon Europe grant ELIAS 101120237), the Ministry of Education, University and Research (FARE grant ML4IP R205T7J2KP) and the Center for Brains, Minds and Machines (CBMM), funded by NSF STC award CCF-1231216. S. V. acknowledges the support of the European Commission (grant TraDE-OPT 861137). The research by S. V. and C. M. has been supported by the MUR Excellence Department Project awarded to Dipartimento di Matematica, Universita di Genova, CUP D33C23001110001. L. R., S. V., M. R., C. M. acknowledge the financial support of the European Research Council (grant SLING 819789). L. R., S. V., C. M. acknowledge the support of the US Air Force Office of Scientific Research (FA8655-22-1-7034), the Ministry of Education, University and Research (grant BAC FAIR PE00000013 funded by the EU - NGEU) and MIUR (PRIN 202244A7YL). M. R., C. M. and S. V. are members of the Gruppo Nazionale per l’Analisi Matematica, la Probabilità e le loro Applicazioni (GNAMPA) of the Istituto Nazionale di Alta Matematica (INdAM). C. M. was also supported by the Programma Operativo Nazionale (PON) “Ricerca 
e Innovazione” 2014–2020. This work represents only the view of the authors. The European Commission and the other organizations are not responsible for any use that may be made of the information it contains.\\
\textbf{Data Availability.} The HTRU2 \cite{htru2,Dua:2019} and CASP \cite{Dua:2019} datasets used in Section \ref{sec:experiements} are available in the following UCI repositories:
\begin{itemize}
    \item HTRU2: \url{https://archive.ics.uci.edu/dataset/372/htru2}
    \item CASP: \url{https://archive.ics.uci.edu/dataset/265/physicochemical+properties+of+protein+tertiary+structure}
\end{itemize}
While California Housing dataset is a standard scikit-learn dataset - see \url{https://inria.github.io/scikit-learn-mooc/python_scripts/datasets_california_housing.html}. Pre-processing operations are described in Appendix \ref{app:flk_exp_details}.\\
\textbf{Code Availability.} The code used to perform the experiments is provided in the supplemental material. Falkon library is available at the following Github repository \url{https://github.com/FalkonML/falkon}. Additional details are provided in Appendix \ref{app:flk_exp_details}.
\section*{Declarations}
\textbf{Conflict of Interest.} The authors have no relevant financial or non-financial interests to disclose. Furthermore,
the authors have no competing or conflict of interests to declare that are relevant to the content of this article
\appendix

\section{Preliminary results}\label{app:aux}
In the appendices, we provide the preliminary results used to prove the main theorems. 

\noindent {\bf Notation.} We use the following filtrations: 
\begin{itemize}
\item $\mathcal{F}_k=\sigma(Z_1,\ldots,Z_{k-1})$;
\item $\mathcal{G}_k=\sigma(P_1, \cdots, P_{k-1})$;
\item $\mathcal{H}_k=\sigma(P_1, Z_1, \cdots, P_{k - 1}, Z_{k - 1})$.
\end{itemize}
To improve readability, we use the following abbreviations: 
\begin{itemize}
    \item $F_k = F(x_k, Z_k)$ and $D_k F_k = D_{(P_k, h_k)} F(x_k, Z_k)$
    \item $f_k= f(x_k)$ and  $\nabla f_k = \nabla f(x_k)$. 
\end{itemize}

\begin{lemma}\label{lem:lip}
Suppose that Assumptions~\ref{ass:lam_smooth}(i)  holds true. If $ \mathbb{E}_Z [\|\nabla F(x,Z)\| ] <+\infty$  for every $x\in\mathbb{R}^d$, then $f$ is differentiable with, for every $x\in\mathbb{R}^d$,
\[
\nabla f(x)=\E_Z[\nabla F(x,Z)]
\]
and $\nabla f$ is $\lambda$-Lipschitz continuous. 
\end{lemma}
\begin{proof}
    Let $x\in\mathbb{R}^d$, $z\in\mathcal{Z}$, $t\in (0,1]$ and let $e_i$ be the $i$-th coordinate vector. The Descent Lemma (see \cite{polyak1987introduction}) yields
\[\frac{\vert F(x+t e_i,z)-F(x,z)\vert}{t} \leq {\vert \langle \nabla F(x,z), e_i \rangle\vert + \frac{t\lambda}{2}} \leq \|\nabla F(x,z)\|+ \frac{\lambda}{2}.
\]
Since $ \mathbb{E}_Z [\|\nabla F(x,Z)\| ] <+\infty$ by assumption, the dominated convergence theorem implies that 
\begin{align*}
    \nabla_{x_i}f(x)=& \lim_{t\to 0} \frac{\E_Z[F(x+t e_i,Z)-F(x,Z)]}{t} \\ =& \E_Z\left[\lim_{t\to 0}\frac{F(x+t e_i,Z)-F(x,Z)}{t}\right]\\ =& \E_Z\left[\nabla_{x_i}F(x,Z) \right]. 
    \end{align*}
    The Lipschitz continuity of $\nabla f$  follows directly  from Assumption~\ref{ass:lam_smooth}(i).
\end{proof}

\noindent The next lemma is a direct consequence of Assumption~\ref{ass:pk_exp}.
\begin{lemma}\label{lemma:pk}
	Let $P\in \R^{d \times l}$ be a random matrix satisfying Assumption~\ref{ass:pk_exp} and let $v \in \R^l$ and $w \in \R^d$.
	Then,
	\begin{equation*}
		\|P v\|^2 \stackrel{\text{a.s.}}{=} \frac{d}{l} \nor{v}^2 \quad \text{and} \quad \E \nor{P ^\intercal w}^2 = \nor{w}^2.	
	\end{equation*} 
\end{lemma}
\begin{proof} 
Since $P^\intercal P=d/\ell I$ almost surely, we have $\|Pv\|^2=\langle v, P^TPv \rangle=(d/l) \|v\|^2$.
Moreover, by the assumption $\E[P^\intercal P]=I$,
\[\E \nor{P^\intercal w}^2 = \E[\langle w, P^\intercal P w\rangle]=\langle w, \E[P^\intercal P] w\rangle = \|w\|^2.\]

\end{proof}
\noindent Next we bound the  error between our surrogate and the stochastic gradient.	
\begin{lemma}[Bound on error norm]\label{lem:err_norm_bound}
	Let $z\in\mathcal{Z}$, $F$ a function satisfying Assumption~\ref{ass:lam_smooth} and $P$ a matrix satisfying Assumption~\ref{ass:pk_exp}. Then, $\forall x \in \R^d$, a.s.
	\begin{equation*}
		\nor{D_{(P, h)}F(x, z) - P^\intercal \nabla F(x, z)}{\leq} \frac{\lambda d h}{2 \sqrt{l}} .
	\end{equation*} 
\end{lemma}
\begin{proof}
	By the definition of the norm, for every $x \in \R^d$ and for every $z\in\mathcal{Z}$, a.s.
	\begin{equation}\label{eqn:nor_err}
		\nor{D_{(P, h)} F(x, z) - P^\intercal \nabla F(x, z)} {=} \sqrt{\sum\limits_{i=1}^l \Big( [D_{(P, h)} F(x, z)]_i - \scalT{\nabla F(x, z)}{p^{(i)}} \Big)^2}.
	\end{equation}
	By the Descent Lemma \cite{polyak1987introduction} and Proposition \ref{lemma:pk}, we have that a.s.
 \begin{eqnarray*}
		 \Big|[D_{(P, h)} F(x, z)]_i - \scalT{\nabla F(x, z)}{p^{(i)}} \Big| && {=} \frac{1}{h}\Big|F(x + h p^{(i)}, z) - F(x, z) - h \scalT{\nabla F(x, z)}{p^{(i)}}\Big|\\
  && \leq  \frac{\lambda h }{2 }\nor{p^{(i)}}^2=\frac{\lambda h d}{2 l}.
	\end{eqnarray*}
	Using this inequality in equation \eqref{eqn:nor_err}, we get the claim. 
\end{proof}
\noindent The above lemma  states that the distance between the finite difference approximations and the directional derivatives is bounded and  goes to $0$ when $h \rightarrow 0$. In the next lemma, we show a bound related to the function value decrease along the iterations.
\begin{proposition}
    [Functional value estimate]\label{prop:abg_inst_fun_dec}
	Under assumptions \ref{ass:lam_smooth}  and \ref{ass:pk_exp} let  $(x_k)$ be the random sequence generated by Algorithm~\ref{alg:szo}. Define $f_k := f(x_k)$ and the sequence
    \begin{equation*}
        C_k := \frac{\lambda  d}{l} \alpha_k^2 G+ \alpha_kh_k^2\Big(\frac{\lambda d }{2\sqrt{l}}\Big)^2\Big(\frac{1}{2}+\frac{\lambda d}{l}\alpha_k\Big).
    \end{equation*}
    Then, for every $k\in\N$,
   \begin{equation}\label{eqn:abg_fv_eq}
        \begin{aligned}        
           \left[\E_{(P_k, Z_k)} [f_{k+1}|\mathcal{H}_k] - f^*\right] - & \left(1 + \frac{\lambda d A}{l}\alpha_k^2\right) \left[f_k - f^*\right] \leq  \frac{\alpha_k}{2} \Big( \frac{2\lambda d B}{l}\alpha_k-1 \Big) \nor{\nabla f_k}^2 + C_k.
        \end{aligned}
	\end{equation}
\end{proposition}
\begin{proof}
Note that, as $x_k$ depends on $P_0, Z_0,\cdots,P_{k-1}, Z_{k-1}$,  $f_k$ is a random variable. 
By the Descent Lemma, Algorithm~\ref{alg:szo} and Proposition \ref{lemma:pk} we have a.s.

	\begin{equation*}
		\begin{aligned}
			f_{k + 1} - f_k &\leq \scalT{\nabla f_k}{x_{k+1} - x_k} + \frac{\lambda}{2} \nor{x_{k+1} - x_k}^2\\
			&\leq %
		\underbrace{-\alpha_k \scalT{P_k^\intercal \nabla f_k}{D_k F_k}}_{\textbf{=:a}} + \underbrace{\frac{\lambda  d}{2l} \alpha_k^2 \nor{D_k F_k}^2}_{\textbf{=:b}}.
		\end{aligned}
	\end{equation*}
	\paragraph*{Bound on a} Denote $e_k := D_k F_k - P_k^\intercal \nabla F_k$. Thus
	\begin{equation*}
		\begin{aligned}
			-\alpha_k \scalT{P_k^\intercal \nabla f_k}{D_k F_k} &= -\alpha_k \scalT{P_k^\intercal \nabla f_k}{e_k} - \alpha_k \scalT{P_k^\intercal \nabla f_k}{P_k^\intercal \nabla F_k}.
		\end{aligned}
	\end{equation*}
	\paragraph*{Bound on b} Adding and subtracting $P_k^\intercal \nabla F_k$ and using Lemma~\ref{lem:err_norm_bound}, we have
	\begin{equation*}
		\begin{aligned}
			\frac{\lambda  d}{2l} \alpha_k^2 \nor{D_k F_k - P_k^\intercal \nabla F_k + P_k^\intercal \nabla F_k}^2 &\leq \frac{\lambda  d}{2l}  \Big(\frac{\lambda d }{2 \sqrt{l}} \Big)^2 \alpha_k^2 h_k^2 + \frac{\lambda  d}{2l} \alpha_k^2 \nor{P_k^\intercal \nabla F_k}^2\\
			&+ \frac{\lambda  d}{l} \alpha_k^2 \scalT{e_k}{P_k^\intercal \nabla F_k}.
		\end{aligned}
	\end{equation*}
We have
	\begin{equation*}
		\begin{aligned}
			f_{k+1} - f_k &\leq -\alpha_k \scalT{P_k^\intercal \nabla f_k}{e_k} - \alpha_k \scalT{P_k^\intercal \nabla f_k}{P_k^\intercal \nabla F_k} + \frac{\lambda  d}{2l}  \Big( \frac{\lambda d }{2\sqrt{l}} \Big)^2 \alpha_k^2 h_k^2\\ 
			&+ \frac{\lambda  d}{2l} \alpha_k^2 \nor{P_k^\intercal \nabla F_k}^2 +\frac{\lambda  d}{l} \alpha_k^2 \scalT{e_k}{P_k^\intercal \nabla F_k }.
		\end{aligned}
	\end{equation*}
	By Cauchy-Schwartz inequality and Lemma~\ref{lem:err_norm_bound}, we get
	\begin{equation*}
		\begin{aligned}
			f_{k+1} - f_k &\leq \alpha_k \vert \scalT{P_k^\intercal \nabla f_k}{e_k}\vert - \alpha_k \scalT{P_k^\intercal \nabla f_k}{P_k^\intercal \nabla F_k}\\
			&\quad + \frac{\lambda  d}{2l}  \Big( \frac{\lambda d}{2\sqrt{l}} \Big)^2 \alpha_k^2 h_k^2 + \frac{\lambda  d}{2l} \alpha_k^2 \nor{P_k^\intercal \nabla F_k}^2\\
			&\quad +\frac{\lambda  d}{l} \Big( \frac{\lambda d}{2\sqrt{l}} \Big) \alpha_k^2 h_k \nor{P_k^\intercal \nabla F_k}.
		\end{aligned}
	\end{equation*}
	From Young's inequality we derive that   $ \vert \scalT{P_k^\intercal \nabla f_k}{e_k}\vert \leq \frac{1}{2}\|P_k^T\nabla f_k\|^2+\frac{1}{2}\Big(\frac{\lambda d}{2\sqrt{l}}\Big)^2h_k^2$ and $\Big( \frac{\lambda d }{2\sqrt{l}} \Big)h_k\|P_k^T\nabla F_k\|\leq \Big( \frac{\lambda d }{2\sqrt{l}} \Big)^2\frac{h_k^2}{2}+\frac{1}{2}\|P_k^T\nabla F_k\|^2$, therefore we have that
	\begin{equation*}
		\begin{aligned}
			f_{k+1} - f_k &\leq \frac{\alpha_k}{2} \nor{P_k^\intercal \nabla f_k}^2 + \frac{\alpha_k}{2} \Big( \frac{\lambda d }{2\sqrt{l}} \Big)^2 h_k^2\\
			&- \alpha_k \scalT{P_k^\intercal \nabla f_k}{P_k^\intercal \nabla F_k} + \frac{\lambda  d}{2l}  \Big( \frac{\lambda d}{2\sqrt{l}} \Big)^2 \alpha_k^2 h_k^2 + \frac{\lambda  d}{2l} \alpha_k^2 \nor{P_k^\intercal \nabla F_k}^2\\
			&+\frac{\lambda  d}{2l} \Big( \frac{\lambda d }{2\sqrt{l}} \Big)^2 \alpha_k^2 h_k^2 + \frac{\lambda  d}{2l} \alpha_k^2 \nor{P_k^\intercal\nabla F_k}^2.
		\end{aligned}
	\end{equation*}
	Observe that $\scalT{P_k^\intercal \nabla f_k}{P_k^\intercal \nabla F_k} = \scalT{P_k P_k^\intercal \nabla f_k}{ \nabla F_k}$. Taking the expectation with respect to $P_k$ conditioned on $\mathcal{H}_k$ and $\sigma(Z_k)$, Proposition~\ref{lemma:pk} implies
	\begin{equation*}
		\begin{aligned}
			\E_{P_k}[f_{k+1}|\mathcal{H}_k, \sigma(Z_k)] - f_k &\leq \frac{\alpha_k}{2} \nor{\nabla f_k}^2 + \frac{\alpha_k}{2} \Big( \frac{\lambda d }{2\sqrt{l}} \Big)^2 h_k^2\\
			&- \alpha_k \scalT{\nabla f_k}{ \nabla F_k}+ \frac{\lambda  d}{l} \Big( \frac{\lambda d}{2\sqrt{l}} \Big)^2 \alpha_k^2 h_k^2 + \frac{\lambda  d}{l} \alpha_k^2 \nor{\nabla F_k}^2.
		\end{aligned}
	\end{equation*}
Taking the conditional expectation with respect to $Z_k$ conditioned on the filtration $\mathcal{H}_k$, we get
    \begin{equation*}
		\begin{aligned}
			\E_{Z_k}[\E_{P_k}[f_{k+1}|\mathcal{H}_k, \sigma(Z_k)]\,|\, \mathcal{H}_k] - f_k &\leq \frac{\alpha_k}{2} \nor{\nabla f_k}^2 + \alpha_kh_k^2\Big(\frac{\lambda d }{2\sqrt{l}}\Big)^2\Big(\frac{1}{2}+\frac{\lambda d}{l}\alpha_k\Big)\\
			&- \alpha_k \| \nabla f_k \|^2 + \frac{\lambda  d}{l} \alpha_k^2 \E_{Z_k}[\nor{\nabla F_k}^2 \, |\, \mathcal{H}_k].
        \end{aligned}
    \end{equation*}
        By Fubini's Theorem, we have
    \begin{equation*}
        \E_{Z_k}[\E_{P_k}[f_{k+1}|\mathcal{H}_k, \sigma(Z_k)]\,|\, \mathcal{H}_k] = \E_{(P_k, Z_k)} [f_{k+1}|\mathcal{H}_k].
    \end{equation*}
    Then, by Assumption \ref{ass:lam_smooth}, we have
    \begin{equation*}
        \begin{aligned}        
           \E_{(P_k, Z_k)} [f_{k+1}|\mathcal{H}_k] - f_k &\leq-\frac{\alpha_k}{2}\nor{\nabla f_k}^2 + \alpha_kh_k^2\Big(\frac{\lambda d }{2\sqrt{l}}\Big)^2\Big(\frac{1}{2}+\frac{\lambda d}{l}\alpha_k\Big)\\
			&+ \frac{\lambda  d}{l} \alpha_k^2 \Bigg( A(f_k - f^*) + B \| \nabla f_k \|^2 + G \Bigg).
        \end{aligned}
	\end{equation*}
    Adding and subtracting $f^*$, we get the claim.
 \end{proof}

 \section{Convex case}\label{app:proofs_abg_cond}
\begin{proposition}[Iterates norm estimate]\label{prop:fej}
	Under Assumptions \ref{ass:lam_smooth}, \ref{ass:pk_exp} and \ref{ass:fun_conv}, suppose also that $\alpha_k\leq \bar{\alpha}$ with $\bar{\alpha} d (A+\lambda B)/ l <1$. Let $(x_k)$ be the random sequence generated by Algorithm~\ref{alg:szo}. For every $k \in \N$, set
 $$D_k := \frac{ d \alpha_k}{l}\Big( 
	\frac{ \lambda^2  d^2}{2l}\alpha_k h_k^2 + 2 G \alpha_k + \frac{\lambda^2 \sqrt{d}}{2}h_k\Big).$$
 Let $\delta\in[0,1-\bar{\alpha} d (A+\lambda B)/l]$ and set $\Delta:=\frac{2}{\lambda}\left[1-\delta- \bar{\alpha}d(A+\lambda B)/l\right]$. Then $\Delta \geq 0$ and, for every $\bar{x} \in \argmin f$, denoting $u_{k}:=\|x_k-\bar{x}\|$, we have that 
 \begin{equation}\label{eqn:abg_fej1}
		\begin{aligned}
			\E_{(P_k, Z_k)}[u_{k + 1}^2 | \mathcal{H}_k] - \left(1 + \frac{d\sqrt{d}}{2l} \alpha_k h_k\right) u_k^2 
			+2\delta\alpha_k(f_k-f^*)+ \Delta \alpha_k \nor{\nabla f_k}^2
			&\leq D_k.
		\end{aligned}
	\end{equation}
\end{proposition}
\begin{proof}
Let $\bar{x}\in\argmin f$. Using the definition of $x_{k+1}$ in Algorithm~\ref{alg:szo}  we derive 
	\begin{equation*}
		\nor{x_{k+1} - \bar{x}}^2 - \nor{x_k - \bar{x}}^2 = \underbrace{\alpha_k^2\nor{ P_k D_k F_k}^2}_{\text{\bf =:a}} +  \underbrace{2\alpha_k \scalT{P_k D_k F_k}{\bar{x} - x_k}}_{\text{\bf =:b}}.
	\end{equation*}
	Denoting $e_k := D_k F_k - P_k^\intercal \nabla F_k$, we bound separately the two terms. 
	To bound {\bf a}, we add and subtract $P_k^\intercal \nabla F_k$ and we use Proposition \ref{lemma:pk}. We get that a.s.
	\begin{equation*}
		\nor{ P_k D_k F_k}^2=\frac{ d}{l} \nor{D_k F_k - P_k^\intercal \nabla F_k + P_k^\intercal \nabla F_k}^2 \leq \frac{2 d }{l} \nor{e_k}^2 + \frac{2 d}{l} \nor{P_k^\intercal \nabla F_k}^2. 
	\end{equation*}
	Moreover, it follows from Lemma~\ref{lem:err_norm_bound} that $ \nor{e_k}^2\leq ({\lambda d}/{2 \sqrt{l}})^2 h_k^2$. To bound {\bf b}, we add and subtract $P_k P_k^\intercal \nabla F_k$. We derive that a.s.
	\begin{equation*}
		\begin{aligned}
			\scalT{P_k D_k F_k}{\bar{x} - x_k} &= 		\scalT{P_k P_k^\intercal \nabla F_k}{\bar{x} - x_k}
			+ \scalT{P_k e_k}{\bar{x} - x_k}\\
			&\leq \scalT{P_k P_k^\intercal \nabla F_k}{\bar{x} - x_k}+  \sqrt{\frac{d}{l}} \Big( \frac{\lambda d}{2 \sqrt{l}} \Big) h_k \nor{\bar{x} - x_k}.
		\end{aligned}
	\end{equation*}
	In the last inequality we used Cauchy-Schwartz, Proposition \ref{lemma:pk} and Lemma~\ref{lem:err_norm_bound}. Now recall the definition of $u_{k}:=\|x_k-\bar{x}\|$. The previous upper-bounds imply that a.s.
	\begin{equation*}
		\begin{aligned}
			u_{k + 1}^2 - u_k^2 &\leq \frac{2 d}{l}  \Bigg( \frac{\lambda d }{2 \sqrt{l}} \Bigg)^2 \alpha_k^2h_k^2+ \frac{2 d}{l} \alpha_k^2 \nor{P_k^\intercal \nabla F_k}^2 + 2\alpha_k \scalT{P_k P_k^\intercal \nabla F_k}{\bar{x} - x_k}\\
			&+ 2 \sqrt{\frac{d}{l}} \Big( \frac{\lambda d }{2 \sqrt{l}} \Big) \alpha_k h_k u_k.
		\end{aligned}
	\end{equation*}
	Taking the expectation with respect to $P_k$ conditioned on $\mathcal{H}_k$ and $\sigma(Z_k)$ and using Proposition~\ref{lemma:pk}, we get that a.s.
	\begin{equation*}
		\begin{aligned}
			\E_{P_k}[u_{k + 1}^2| \mathcal{H}_k, \sigma(Z_k)] - u_k^2 &\leq \frac{2 d}{l} \Bigg( \frac{\lambda d}{2 \sqrt{l}} \Bigg)^2 \alpha_k^2 h_k^2+ \frac{2 d}{l} \alpha_k^2 \nor{\nabla F_k}^2 + 2\alpha_k \scalT{\nabla F_k}{\bar{x} - x_k}\\
			&+ 2 \sqrt{\frac{d}{l}} \left(\frac{\lambda d}{2 \sqrt{l}} \right)\alpha_k h_k u_k.
		\end{aligned}
	\end{equation*}
	Define
 \begin{equation} \label{eq:defck}
		W_k :=  \frac{\lambda^2  d^3}{2l^2}\alpha_k^2 h_k^2 + \frac{2 d G}{l} \alpha_k^2 \quad \text{and} \quad \rho_k :=  \frac{\lambda d \sqrt{d}}{l} \alpha_k h_k.
	\end{equation}
 Taking the expectation with respect to $Z_k$ conditioned on $\mathcal{H}_k$, using Assumption \ref{ass:lam_smooth} (iii) and that $\alpha_k \leq \bar{\alpha}$, we derive that
\begin{equation*}
		\begin{aligned}
			\E_{Z_k} [\E_{P_k} [u_{k + 1}^2 | \mathcal{H}_k, \sigma(Z_k)] | \mathcal{H}_k] - u_k^2 &\leq  W_k +\frac{2  d\bar{\alpha}}{l} \alpha_kA(f_k-f^*)+ \frac{2  d\bar{\alpha} }{l} \alpha_k B \nor{\nabla f_k}^2 \\
			&\quad  + 2\delta \alpha_k \scalT{\nabla f_k}{\bar{x} - x_k} + 2\alpha_k(1-\delta) \scalT{\nabla f_k}{\bar{x} - x_k}+\rho_k u_k,
		\end{aligned}
	\end{equation*}
where the previous inequality holds for every $\delta$ and so in particular for $\delta$ chosen between $0$ and $1-d (A+\lambda B)\bar{\alpha}/l$ as in the assumptions. Note that $u_k^2$ is integrable (the proof follows the same line of \cite[Remark 23]{kozak2021zeroth}). By Fubini's Theorem we have 
    \begin{equation*}
 \E_{Z_k} [ \E_{P_k}[u_{k + 1}^2 | \mathcal{H}_k, \sigma(Z_k)] | \mathcal{H}_k] = \E_{(P_k, Z_k)}[u_{k + 1}^2 | \mathcal{H}_k].       
    \end{equation*}
Moreover, by convexity of $f$ we have both that $A(f_k-f^*)\leq -A\scalT{\nabla f_k}{\bar{x}-x_k}$ and that $\delta \scalT{\nabla f_k}{\bar{x}-x_k}\leq \delta(f^*-f_k)$. Then,
\begin{equation*}
		\begin{aligned}
			\E_{(P_k, Z_k)}[u_{k + 1}^2 | \mathcal{H}_k] - u_k^2 + & 2\delta \alpha_k(f_k-f^*) \leq W_k+\rho_k u_k+ \frac{2  d\bar{\alpha}}{l} \alpha_k B \nor{\nabla f_k}^2+\\
   & + 2\alpha_k \left[ 1-\delta -\frac{d\bar{\alpha} }{l} A \right] \scalT{\nabla f_k}{\bar{x} - x_k}.
		\end{aligned}
	\end{equation*}
Since $\delta\leq  1-d (A+\lambda B) \bar{\alpha}/l\leq 1-d \bar{\alpha}A/l$, then $1-\delta-d\bar{\alpha}A/l\geq  0$. Then, the $(1/\lambda)-$cocoercivity of $\nabla f$ and the definition of $\Delta$ yield 
 \begin{equation}\label{est_1}
		\begin{aligned}
			\E_{(P_k, Z_k)}[u_{k + 1}^2 | \mathcal{H}_k] - u_k^2 + & 2\delta \alpha_k(f_k-f^*) +\Delta \alpha_k \|\nabla f_k\|^2 \leq W_k+\rho_k u_k.
		\end{aligned}
	\end{equation}   
Now, using  Young's inequality with parameter $\tau_k$, we obtain
 \begin{equation}\label{eqn:fej2}
		\begin{aligned}
			\rho_k u_k
   \leq \frac{1}{4\tau_k} \rho_k^2  + \tau_k u_k^2.
		\end{aligned}
	\end{equation} 
 Finally, setting $\tau_k=d\sqrt{d} \alpha_k h_k/(2l)$, we get the claim.
 \end{proof}
\noindent Lemma~\ref{prop:fej} is the extension to the stochastic setting of \cite[Lemma 2]{kozak2021zeroth}. The difference between equation \eqref{eqn:abg_fej1} and \cite[Lemma 2]{kozak2021zeroth} is the presence of extra terms related to the stochastic information combined with the \eqref{eq:ABG} condition. 
 \begin{proposition}\label{prop_convx}
	Let Assumptions  \ref{ass:lam_smooth},
    \ref{ass:pk_exp}, \ref{ass:step_disc} and \ref{ass:fun_conv} hold and suppose also that $\alpha_k\leq \bar{\alpha}$ with $\bar{\alpha} d (A+\lambda B) / l <1$. Let $(x_k)$ be the random sequence generated by Algorithm~\ref{alg:szo}. Then the sequence $\nor{x_{k} - \bar{x}}$ is a.s. convergent for every $\bar{x}\in \argmin f$, a.s.
	\begin{equation*}
	\alpha_k \nor{\nabla f_k}^2 \in \ell^1, \quad \alpha_k (f_k - f^*) \in \ell^1,  \qquad \lim\limits_{k} f_k = f^*
	\end{equation*}
	and there exists a random variable $x^*$ taking values in $\argmin f$ such that $x_k \rightarrow x^*$ a.s.
 \end{proposition}
 \begin{proof}
We use the estimate provided by Proposition \ref{prop:fej} with a value of $\delta>0$ such that $\Delta > 0$. By Assumption~\ref{ass:step_disc}, $d\sqrt{d} \alpha_k h_k/(2l)$ and $D_k$ belong to $\ell^1$. 
 Then Robbins-Siegmund Theorem \cite{ROBBINS1971233} implies that, for every $\bar{x}\in \argmin f$, the sequence $\nor{x_{k} - \bar{x}}$ is a.s. convergent for $k \to +\infty$ and that $\alpha_k \nor{\nabla f_k}^2$ and $\alpha_k(f_k-f^*)$ belong to $\ell^1$ a.s.	 By Assumption~\ref{ass:step_disc}, we have also that the sequence $\alpha_k$ does not belong to $\ell^1$. Then, $\alpha_k(f_k-f^*)\in \ell^1$ a.s. implies that a.s.
	\begin{equation}\label{eqn:liminf}
		\liminf\limits_{k} f_k = f^*.
	\end{equation}
 Now consider the estimate in Proposition~\ref{prop:abg_inst_fun_dec}. By Assumption \ref{ass:step_disc}, $\lambda d A \alpha_k^2/l$ and $C_k$ belong to $\ell^1$. Moreover, we know that $\alpha_k^2 \nor{\nabla f_k}^2\leq \bar{\alpha} \alpha_k \nor{\nabla f_k}^2 \in \ell^1$ a.s. Then, Robbins-Siegmund Theorem implies that $f_k$ is a.s. convergent for $k\to +\infty$. Jointly with equation~\eqref{eqn:liminf}, it yields that a.s. $$\lim_k  f_k = f^*.$$ 
 In particular, every a.s. cluster point of the sequence $(x_k)$ is a.s. a minimizer of the function $f$. Recall also that $\nor{x_{k} - \bar{x}}$ is a.s. convergent for every $\bar{x} \in \argmin f$. Then, by Opial's Lemma,  there exists a random variable $x^* $ taking values in $\argmin f$ such that $x_k \rightarrow x^*$ a.s.
\end{proof}
\noindent Now, to get convergence rates from the previous energy estimation, we need to upper bound the quantity $\sqrt{\E[\nor{x_k-\bar{x}}^2]}$ for $\bar{x} \in \argmin f$. 
\begin{lemma}\label{lem:bih}
	Let Assumptions  \ref{ass:lam_smooth}, 
 \ref{ass:pk_exp} and \ref{ass:fun_conv} hold and suppose also that $\alpha_k\leq \bar{\alpha}$ with $\bar{\alpha} d (A+\lambda B)/ l <1$. Let $(x_k)$ be the random sequence generated by Algorithm~\ref{alg:szo}. Let $\bar{x} \in \argmin f$ and define the sequences $W_k, \rho_k$ as in \eqref{eq:defck}. Moreover, define
 \begin{equation}\label{eq:bih1}
	    \begin{aligned}
	        S_k := \nor{x_0 - \bar{x}}^2 + \sum\limits_{j=0}^{k} W_j.
	    \end{aligned}
	\end{equation}
 Then we have that, for all $k \in \N$,
 \begin{equation}\label{eq:bih0}
		\sqrt{\E[\nor{x_k - \bar{x}}^2]} \leq \sqrt{S_{k}} + \sum\limits_{j=0}^{k}\rho_j.
	\end{equation}
\end{lemma}
\begin{proof}
	Choose $\delta=0$ in \eqref{est_1} and take the total expectation. Let $v_k := \sqrt{\E[\nor{x_{k} - \bar{x}}^2]}$ and note that, by Jensens's inequality, $\E[u_k]\leq v_k$. Then, denoting $\Delta_0:= \frac{2}{\lambda}\Big(1  - \frac{\bar{\alpha}d(A+\lambda B)}{l} \Big)$, we get
   \begin{equation*}
		v_{k+1}^2 - v_{k}^2 + \Delta_0 \alpha_k \E[\nor{\nabla f_k}^2] \leq W_k + \rho_k v_k.
	\end{equation*}
	Summing the previous inequality from $0$ to $k$ we derive
	\begin{equation*}
		v_{k+1}^2 \leq S_k  + \sum\limits_{j=0}^k\rho_j v_j.
	\end{equation*}
	Now we observe that $\rho_k$ is non-negative, $S_k$ is non-decreasing and $S_0 \geq v_0^2$. We derive \eqref{eq:bih0} by the (discrete) Bihari's Lemma \cite[Lemma 9.8]{kozak2021zeroth}.
\end{proof}

\begin{theorem}\label{th_convex}
	Let Assumptions  \ref{ass:lam_smooth}, \ref{ass:pk_exp} and \ref{ass:fun_conv} hold and suppose also that $\alpha_k\leq \bar{\alpha}$ with $\bar{\alpha} d (A+\lambda B)/ l <1$. Let $(x_k)$ be the random sequence generated by Algorithm~\ref{alg:szo} and $(\hat{x}_k)$ the sequence of the ergodic iterates. Then, for every $\delta \in \left(0, 1-\bar{\alpha}d (A+\lambda B)/l\right]$ and for every $\bar{x} \in \argmin f$, we have
\begin{align}\label{eq:prerate00}
&\E[f(\hat{x}_k) - f^*]  \leq \frac{1}{2 \delta \sum\limits_{j=0}^k\alpha_j} \left[S_k+  \sum\limits_{j=0}^k \rho_j \left(\sqrt{S_{j}} + \sum\limits_{i=0}^{j}\rho_i\right)\right],
\end{align}
where the sequences $W_k, \rho_k$ and $S_k$ are defined as in \eqref{eq:defck} and \eqref{eq:bih1}.
\end{theorem}

\begin{proof}
     Using the same notation as in the proof of Proposition~\ref{prop:fej}, see \eqref{eq:defck}, 
     it follows from equation~\eqref{est_1} that
 \begin{equation}\label{eq:est_2}
		\begin{aligned}
			\E_{(P_k, Z_k)}[u_{k + 1}^2 | \mathcal{H}_k] - u_k^2 + & 2 \delta\alpha_k(f_k-f^*)  \leq W_k+\rho_k u_k.
		\end{aligned}
	\end{equation}  
 Taking the total expectation and denoting again $v_k := \sqrt{\E[\nor{x_k - x^*}^2]}$, we get
\begin{equation*}
\begin{aligned}
			v_{k + 1}^2  - v_k^2 + & 2\delta \alpha_k\E(f_k-f^*)  \leq W_k+\rho_k v_k.
		\end{aligned}
\end{equation*}
Summing from $0$ to $k$, we have
\begin{align*}
	v_{k+1}^2  + 2\delta\sum\limits_{j=0}^k\alpha_j \E[(f_j - f^*)] \leq v_0^2 + \sum\limits_{j=0}^k W_j+ \sum\limits_{j=0}^k \rho_j v_j.
\end{align*}
Equation \eqref{eq:bih0} of Lemma~\ref{lem:bih} and $v_{k+1}^2\geq 0$ yield
\begin{equation}\label{eq:prerate}
2\delta\sum\limits_{j=0}^k\alpha_j \E[(f_j - f^*)] \leq  S_k +  \sum\limits_{j=0}^k \rho_j \left(\sqrt{S_{j}} + \sum\limits_{i=0}^{j}\rho_i\right).
\end{equation}
Moreover, by convexity, 
\begin{equation*}
	\E[f(\hat{x}_k) - f^*]  \leq  \frac{\sum\limits_{j=0}^k\alpha_j \E[f(x_j) - f^*]}{\sum\limits_{j=0}^k\alpha_j},
\end{equation*}
which concludes the proof.
\end{proof}

\subsection{Proof of Corollary \ref{cor:conv}}\label{app:proof_cor_conv}
Notice that, for both cases $A+\lambda B >0$ with $\bar{\alpha}=l\eta /[d (A+\lambda B)]$ and $A+\lambda B =0$ with $\bar{\alpha}>0$, the step-size sequence satisfies $\alpha_k\leq \bar{\alpha}$ with $\bar{\alpha} d (A+\lambda B) /l<1$. So the assumptions of Theorem \ref{th_convex} hold and we have that, for every $\delta \in (0,1-\bar{\alpha} d (A+\lambda B)/l]$,
\begin{equation*}
        \begin{aligned}
        \E[f(\hat{x}_k) - f^*]  &\leq \frac{1}{2\delta {\sum_{j=0}^k\alpha_j}} \Bigg[ S_k + \sum\limits_{j=0}^k \rho_j \Bigg( \sqrt{S_j}  +\sum\limits_{i=0}^{j} \rho_i \Bigg) \Bigg].            
        \end{aligned}
    \end{equation*}
Choose $\delta =1-d (A+\lambda B)\bar{\alpha}/l$, namely $\delta=1-\eta$ if $A+\lambda B>0$ and $\delta=1$ if $A+\lambda B =0$. In both cases, $\delta$ is a constant independent from $d, l, \bar{\alpha}$ and $\bar{h}$. By the assumptions, we know that the sequences $a_k^2\xi_k^2$, $a_k^2$ and $a_k \xi_k$ are summable. In particular there exists a constant $c>0$, depending only on $a_k$ and $\xi_k$, such that  
    \begin{equation*}
        \begin{aligned}
            \sum\limits_{j = 0}^k W_j &\leq c \left(\frac{\lambda^2 d^3 \bar{\alpha}^2 \bar{h}^2}{2l^2} +  \frac{2G d\bar{\alpha}^2}{l}\right);\\
            \sum\limits_{j = 0}^k \rho_j &\leq c \frac{\lambda d\sqrt{d} \bar{\alpha} \bar{h}}{l};\\
            S_k &\leq \|x_0 - \bar{x}\|^2 + c\left( \frac{\lambda^2 d^3 \bar{\alpha}^2 \bar{h}^2}{2l^2} + \frac{2G d\bar{\alpha}^2}{l}\right);\\
            \sum\limits_{j = 0}^k \rho_j \sqrt{S_j} &\leq c \left( \|x_0 -\bar{x}\| \frac{\lambda d\sqrt{d} \bar{\alpha} \bar{h} }{l} + \frac{\lambda^2 d^3 \bar{\alpha}^2 \bar{h}^2}{\sqrt{2}l^2}+ \frac{\sqrt{2} \lambda \sqrt{G} d^2 \bar{\alpha}^2 \bar{h}}{l \sqrt{l}}\right);\\
            \sum\limits_{j = 0}^k \rho_j \sum\limits_{i = 0}^j \rho_i &\leq c \frac{\lambda^2 d^3 \bar{\alpha}^2 \bar{h}^2}{l^2},
        \end{aligned}
    \end{equation*}
 from which we get the inequality in \eqref{ineqqq}.
    Now consider the case $A+\lambda B>0$, in which $\bar{\alpha}=l \eta/[d(A+\lambda B)]$, where $0<\eta<1$ is an adimensional constant. Then, for some adimensional constant $\tilde{c}>0$, we have that
    \begin{equation*}
        \begin{aligned}
        \E[f(\hat{x}_k) - f^*]  &\leq \frac{\tilde{c} d (A+\lambda B)}{l {\sum_{j=0}^k a_j}} \Bigg[ \|x_0 - \bar{x}\|^2 + \frac{\lambda^2 d \bar{h}^2}{(A+\lambda B)^2} +  \frac{G l}{d(A+\lambda B)^2}\\
        &+ \|x_0 -\bar{x}\| \frac{\lambda \bar{h} \sqrt{d}}{(A+\lambda B)} + \frac{\lambda \sqrt{G} \sqrt{l} \bar{h}}{ (A+\lambda B)^2}\Bigg],
        \end{aligned}
    \end{equation*}
    from which we conclude the inequality in \eqref{ineq1}.\\
    \ \\
    Now consider the case $A+\lambda B=0$ and $\bar{\alpha}=\sqrt{l/d}$. Starting again from \eqref{ineqqq}, we conclude the inequality in \eqref{ineq2}.

\section{Non-convex case} \label{app:nonconv}
In this section we prove the results under the PL condition.
\subsection{Proof of Theorem~\ref{thm:abg_pl_rates}}
\label{app:nonconv1}
\begin{proof}
    We start from inequality in Proposition~\ref{prop:abg_inst_fun_dec}. Using the assumption $\alpha_k \leq \bar{\alpha}$, we have that, for every $k\in \N$,
\begin{equation*}
        \begin{aligned}        
           \left[\E_{(P_k, Z_k)} [f_{k+1}|\mathcal{H}_k] - f^*\right] \ \leq \ & \left(1 + \frac{\lambda d A}{l}\alpha_k^2\right) \left[f_k - f^*\right] -\frac{\alpha_k}{2} \Big( 1-\frac{2\lambda d B}{l}\alpha_k \Big) \nor{\nabla f_k}^2 + C_k\\
           \leq \ & \left(1 + \frac{\lambda d A}{l}\bar{\alpha}\alpha_k\right) \left[f_k - f^*\right] -\frac{\alpha_k}{2} \Big( 1-\frac{2\lambda d B}{l}\bar{\alpha} \Big) \nor{\nabla f_k}^2 + C_k.
        \end{aligned}
	\end{equation*}
The assumption on $\bar{\alpha}$ implies that $2d\lambda B \bar{\alpha}/l \leq 1$. Using Assumption \ref{ass:pl}, we get

\begin{equation*}
        \begin{aligned}        
           \left[\E_{(P_k, Z_k)} [f_{k+1}|\mathcal{H}_k] - f^*\right] \leq \left(1 + \frac{\lambda d A}{l}\bar{\alpha}\alpha_k\right) \left[f_k - f^*\right]   -\frac{\alpha_k}{2} \Big( 1-\frac{2\lambda d B}{l}\bar{\alpha} \Big) \gamma \left[f_k - f^*\right] + C_k.
        \end{aligned}
	\end{equation*}
 Taking the full expectation and rearranging the terms, we conclude the claim.
\end{proof}
\subsection{Proof of Corollary \ref{cor:nonconv}}\label{app:nonconv2}
  By Theorem \ref{thm:abg_pl_rates},
     \begin{equation*}
        \begin{aligned}        
           \E \left[f_{k+1}\right] - f^* \leq & \left[1-\alpha_k \Big( \frac{\gamma}{2}-\frac{d\lambda\left(A+\gamma B\right)}{l}\bar{\alpha} \Big) \right] \left[\E \left[f_k\right] - f^*\right]\\
           &+ \frac{\lambda  d}{l} \alpha_k^2 G+ \alpha_k h_k^2\Big(\frac{\lambda d }{2\sqrt{l}}\Big)^2\Big(\frac{1}{2}+\frac{\lambda d}{l}\alpha_k\Big).
        \end{aligned}
	\end{equation*}
    Since $\alpha_k \leq \bar{\alpha}$,
     \begin{equation*}
        \begin{aligned}        
           \E \left[f_{k+1}\right] - f^* \leq & \left[1-\alpha_k \Big( \frac{\gamma}{2}-\frac{d\lambda\left(A+\gamma B\right)}{l}\bar{\alpha} \Big) \right] \left[\E \left[f_k\right] - f^*\right]\\
           &+ \frac{\lambda  d}{l} \alpha_k^2 G+ \alpha_k h_k^2\Big(\frac{\lambda d }{2\sqrt{l}}\Big)^2\Big(\frac{1}{2}+\frac{\lambda d}{l} \bar{\alpha}\Big).
        \end{aligned}
	\end{equation*}
    Set  $C_1 = \left(2\gamma - \frac{d \lambda (A + \gamma B)}{l} \bar{\alpha} \right)\bar{\alpha}$ and $C_2 = \left( \frac{\lambda d \bar{\alpha}^2 G}{l} + \left( \frac{\lambda d}{2\sqrt{l}} \right)^2 \left[ 1 + \frac{\lambda d \bar{\alpha}}{l}\right]\bar{\alpha} h^2 \right)$. Then
     \begin{equation*}
        \begin{aligned}        
           \E \left[f_{k+1}\right] - f^* \leq & \left[1- \frac{C_1}{k^{\theta}} \right] \left[\E \left[f_k\right] - f^*\right] + \frac{C_2}{k^{2\theta}},
        \end{aligned}
	\end{equation*}
     The statement then follows by the Chung's Lemma \cite[Lemma 4]{chung}.
\section{Experiment Details}\label{app:experiment_details}\label{app:changing_l}
Here we report details on the experiments performed. We implemented every script in Python3 and  used scikit-learn~\cite{scikit-learn}, NumPy~\cite{harris2020array} and PyTorch~\cite{NEURIPS2019_9015} libraries.
\subsection{Synthetic experiments}\label{app:syn_exp_details}
In this appendix, we provide details about synthetic experiments presented in Section~\ref{sec:experiements}. We performed four experiments on functions that satisfy Assumptions \ref{ass:fun_conv} and/or \ref{ass:pl}. Specifically, we considered a strongly convex, a PL convex, a convex and a PL non-convex function. The  dimension of the domain is set to $d = 100$ for strongly convex (F1), PL convex (F2), and PL non-convex (F3) targets. In particular, these functions are defined in Table \ref{tab:syn_fun}.
\begin{table}[h!]
	\caption{Target function used in synthetic experiments.}
	\label{tab:syn_fun}	
	\centering
	\begin{tabular}{l l l}
		
		\toprule
		
		Name & Function  & Details\\
		
		\midrule
		F1 & $f(x) = (1/d) \nor{A x}^2$ & $A \in \R^{d \times d}$ full rank\\
		F2 & $f(x) = (1/d) \nor{A x}^2$ & $A \in \R^{d \times d}$ rank deficient\\
		F3 & $f(x) = (1/d) \nor{A x}^2 + 3\sin^2 (c^\intercal x)$ & $A \in \R^{d \times d}$ and $Ac = c$\\
		
		\hline
	\end{tabular}
\end{table}
Both three functions have $x^* = [0, \cdots, 0]^\intercal \in \argmin f$. Moreover, for every function, the matrix $A$ is built with random normal entries i.e. $A_{i,j} \sim \mathcal{N}(0, 1)$ for every $i,j$. 
In these experiments, at each time step $k$, Algorithm~\ref{alg:szo} observes a $z_k$ which represents a row of the matrix $A$ (in all the three cases) and performs the iteration described in %
equation \eqref{eqn:szo}.
For direct-search methods, we used the sufficient decrease condition (instead of the simple decrease condition for STP\cite{stp}) with forcing function $\rho(\alpha) := 10 \alpha^2 \nor{p}^2$, where $p$ is the direction selected. Moreover, opportunistic polling is used i.e. the algorithms do not check any other poll directions as soon as the sufficient decrease condition is satisfied. The initial step-size is set to $\alpha_0 = 1.0$ and for STP, the step-size is decreased as $\alpha_k = \alpha_0/(k + 1)$. Expansion and contraction factors are set to $2$ and $0.9$ respectively. The maximum step size is set to $20$ and the minimum to $10^{-3}$ for the strongly convex and PL convex functions and to $20^{-10}$ for the non-convex target. For S-SZD, we summarize the choice of the parameters in Table \ref{tab:sszd_params_illustrative}.  
For the first experiment of Section \ref{sec:experiements} (i.e. Figure \ref{fig:changing_l_1}), we set $\alpha_k$ and  $h_k$  as $5 \times 10^{-3} \frac{l}{d} k^{-(1/2 + 10^{-10})}$ and $ 10^{-7}/\sqrt{k}$ respectively.
\begin{table}[ht]
	\caption{Parameter of S-SZD for illustrative experiments.} \label{tab:sszd_params_illustrative}
	\centering
	\begin{tabular}{l l l}
		\toprule
		Function  &  stepsize $\alpha_k$  & discretization $h_k$\\
		\midrule
		F1 & $0.3 \frac{l}{d} k^{-(1/2 + 1e-10)} $ & $10^{-7} k^{-(1/2 + 1e-10)}$ \\
		F2 & $0.4 \frac{l}{d} k^{-(1/2 + 1e-10)} $& $10^{-7} k^{-(1/2 + 1e-10)}$ \\
		F3 & $0.5 \frac{l}{d} k^{-(1/2 + 1e-10)} $ & $10^{-7} k^{-(1/2 + 1e-10)}$ \\
		\hline
	\end{tabular}
\end{table}

\subsection{Falkon tuning experiments details}\label{app:flk_exp_details}
In this appendix, we describe how datasets are preprocessed, split and how parameters of the optimizers are selected to perform Falkon experiments. As indicated in Section~\ref{sec:experiements}, we focus on tuning kernel and regularization parameters. %
Given a dataset $D = \{(x_i, y_i)\}_{i=1}^n$ where $y_i = f(x_i) + \epsilon_i$, we want to approximate $f$ with a surrogate $f_w$ which depends on some parameters $w$. To avoid overfitting,  we split the dataset in two, the first part including the data from 1 to $n_1$ (training) and the second one including the remaining points (validation). The training part (composed of the $80\%$ of the points) is used to build $f_w$. The validation part (composed of the remaining $20\%$ of the points) is used to choose hyperparameters $\lambda$ and $\sigma=(\sigma_1,\ldots,\sigma_d)$. Specifically, the problem we want to solve is the following
\begin{equation}\label{eqn:flk_ho0}
    w_{\lambda,\sigma}^* \in \argmin\limits_{w \in W} \sum\limits_{i = 1}^{n_1} (f_w(x_i,\sigma) - y_i)^2 + \lambda \|w\|^2 
    \end{equation}
    \begin{equation}\label{eqn:flk_ho}
    (\lambda^*,\sigma^*) \in \argmin\limits_{\lambda \in \R_{+},\,\sigma\in \R^d_+} \sum\limits_{i = n_1+1}^{n} (f_{w^{*}_{\lambda,\sigma}}(x_i,\sigma) - y_i)^2 + \lambda \|w^{*}_{\lambda,\sigma}\|^2 
\end{equation}
The problem we solve using S-SZD is the one to select $\sigma$ and $\lambda$.
In experiments in Section \ref{sec:experiements}, we modelled $f_w$ with Falkon \cite{rudi2017falkon,falkonlibrary2020}. 
\subsection*{Preprocessing}
Before training Falkon, data are standardized to zero mean and unit standard deviation and, for binary classification problems, labels are set to be $-1$ and $1$. After preprocessing, datasets are split into training and test parts. %
Specifically, for each dataset, we used $80\%$ as training set (use to tune parameters) and $20\%$ as test set (used to evaluated the model) - see Table \ref{tab:datasets}. %
\subsection*{Falkon Parameters}
Falkon has different parameters: a kernel function, a number of points (\Nystrom{} centers) used to compute the approximation and a regularization parameter. We used a Gaussian kernel $k$ with many length-scale parameters $\sigma_1,\cdots,\sigma_d$ with $d$ number of features of the dataset which is defined as follows
\begin{equation*}
	k(x, x^\prime) = e^{- \frac{1}{2}x \Sigma^{-1} x^{\prime}} \qquad \text{with} \qquad \Sigma = \diag{(\sigma^2_1,\cdots, \sigma^2_d)}.
\end{equation*}  
\Nystrom{} centers are uniformly randomly selected and the number of centers $M$ is fixed as the square root of the number of points in the training part %
while the other parameters are optimized.
Falkon library \cite{falkonlibrary2020,falkonhopt2022} used can be found at the following URL: \url{https://github.com/FalkonML/falkon}. Code for Direct search methods is inspired by the implementation proposed in~\cite{prob_ds_sketched} which can be downloaded at the following URL: \url{https://github.com/lindonroberts/directsearch}.

\begin{table}[h]
	\caption{Datasets used to perform FALKON tuning experiments and S-SZD parameters.}	
	\label{tab:datasets}\label{tab:flk_prms}
	\centering
	\begin{tabular}{lccccc}
		
		\toprule
		
		Name     &  $d$ & Training size & Test size & $\alpha_k$ & $h_k$\\
		
		\midrule
		
		HTRU2 \cite{htru2,Dua:2019} & 8 & 14316 & 3580 & $15 \frac{l}{d} k^{-(1/2 + 10^{-10})}$ & $10^{-3}/\sqrt{k}$\\
		CASP \cite{Dua:2019} & 9  & 36583 & 9146 &  $50 \frac{l}{d}k^{-(1/2 + 10^{-10})}$ & $10^{3}/\sqrt{k}$\\
		California Housing & 8 & 16512 & 4128 &  $50 \frac{l}{d} k^{-(1/2 + 10^{-10})}$ & $10^{-5}/\sqrt{k}$\\
		\hline
	\end{tabular}
\end{table}

\subsection*{Parameters of optimizers} For DS optimizers, we used $2$ as expansion factor and $0.5$ as contraction factor. %
Maximum and minimum stepsize are is set to $100$ and $10^{-9}$. %
The initial step size is set to $1$ for CASP dataset, $5$ for California Housing and $10$ for HTRU2. The number of directions used for direct-search methods is $2$. For reduced ProbDS \cite{prob_ds_sketched}, we used sketched matrices with size $d / 2$. %
Direction matrices of S-SZD are generated s.t. Assumption~\ref{ass:pk_exp} hold - see Appendix \ref{app:dir_gen}. The choice of parameters is summarized in Table \ref{tab:flk_prms}. %
In Table~\ref{tab:flk_tuning_te}, we report the mean and standard deviation of test error, obtained with Falkon tuned with different optimizers.
 \begin{table}[ht]
 	\caption{Mean $\pm$ standard deviation of test error (MSE) over 5 runs.}
	
 	\label{tab:flk_tuning_te}
	
 	\centering
	
 	\begin{tabular}{lccc}
		
 		\toprule
		
 		Algorithm & HTRU2 & CASP  &  California Housing \\
		
 		\midrule
		
 		STP \cite{stp} & $0.4246 \pm 0.0353$ & $0.1072 \pm 0.0090$ &  $0.0688 \pm 0.0023$    \\
		SCD  & $0.7680 \pm 0.0146$ & $0.1220 \pm 0.0051$ &  $0.0660 \pm 0.0048$    \\
		Deterministic FD & $0.4056 \pm 0.0694$ & $0.0840 \pm 0.0001$ & $0.0587 \pm 0.0007$\\
		Ghadimi \& Lan \cite{ghadimi_lam} & $0.7958 \pm 0.0028$ & $0.1211 \pm 0.0050$ & $0.0648 \pm 0.0020$\\
		Duchi et al \cite{sto_md} & $0.3987 \pm 0.1040$ & $0.0840 \pm 0.0001$ & $0.0589 \pm 0.0005$\\
        Flaxman et al \cite{flaxman2005online} & $0.8059 \pm 0.0007$ & $0.1047 \pm 0.0054$ & $0.0641 \pm 0.0018$\\
        Berahas et al \cite{berahas2022theoretical} & $0.3938 \pm 0.0989$ & $0.0840 \pm 0.0002$ & $0.0586 \pm 0.0005$\\
 		ProbDS spherical \cite{prob_ds}  & $0.3745 \pm 0.0069$ & $0.0877 \pm 0.0034$ & $0.0591 \pm 0.0014$  \\
 		ProbDS orthogonal \cite{prob_ds} & $0.3786 \pm 0.0092$  & $0.0862 \pm 0.0023$  &  $0.0587 \pm 0.0006$\\
 		ProbDS-RD spherical \cite{prob_ds_sketched}   & $0.3735 \pm 0.0062$  & $0.1005 \pm 0.0107$ & $0.0584 \pm 0.0005$  \\	
 		ProbDS-RD orthogonal \cite{prob_ds_sketched}     & $0.3757 \pm 0.0070$  & $0.0869 \pm 0.0027$  & $0.0587 \pm 0.0010$ \\
 		S-SZD $(l = d/2)$   & $0.3726 \pm 0.0340$ & $0.0852 \pm 0.0012$ & $0.0598 \pm 0.0011$  \\		
 		S-SZD $(l = d)$     & $0.3835 \pm 0.0355$ & $0.0840 \pm 0.0001$ & $0.0587 \pm 0.0008$ \\
		
 		\hline
		
 	\end{tabular}
	
 \end{table}
From Table~\ref{tab:flk_tuning_te}, we  observe that Falkon with parameters tuned with S-SZD obtains similar performances  to the one obtained using with DS methods (except STP).

\subsection{Comparison with Variance Reduced Methods}
In this appendix, we compare Algorithm \ref{alg:szo} with ZO-SVRG-Coord \cite{NEURIPS2018_ba9a56ce} and ZO-SVRG-Rand-Coord \cite{pmlr-v97-ji19a}. These two algorithms exploit variance reduction technique to improve rates. However, these methods can be used only if the target function $f$ is a finite-sum of functions. We compare the performance of these algorithms in solving a strongly convex, a PL convex and a PL non convex optimization problems described in Appendix \ref{app:syn_exp_details} considering the dimension of the search space $d = 50$. We also propose the following heuristic stepsize (imp stepsize) for S-SZD: given $\tau, \delta > 0$ and $\alpha_1 > \alpha_0 > 0$
\begin{equation}\label{app:imp_step}
    \alpha_k = \left\{ \begin{array}{cc}
        \alpha_0 k^{-(1/2 + \delta)} & \text{ if }k < \tau  \\
         \alpha_1 k^{-(1/2 + \delta)} & 
    \end{array}
    \right.
\end{equation}
We fix a budget of $100000$ function evaluations. We repeated the experiments $10$ times and we reported the mean and the standard deviation.
\begin{figure}[H]
    \centering
    \includegraphics[width=0.8\linewidth]{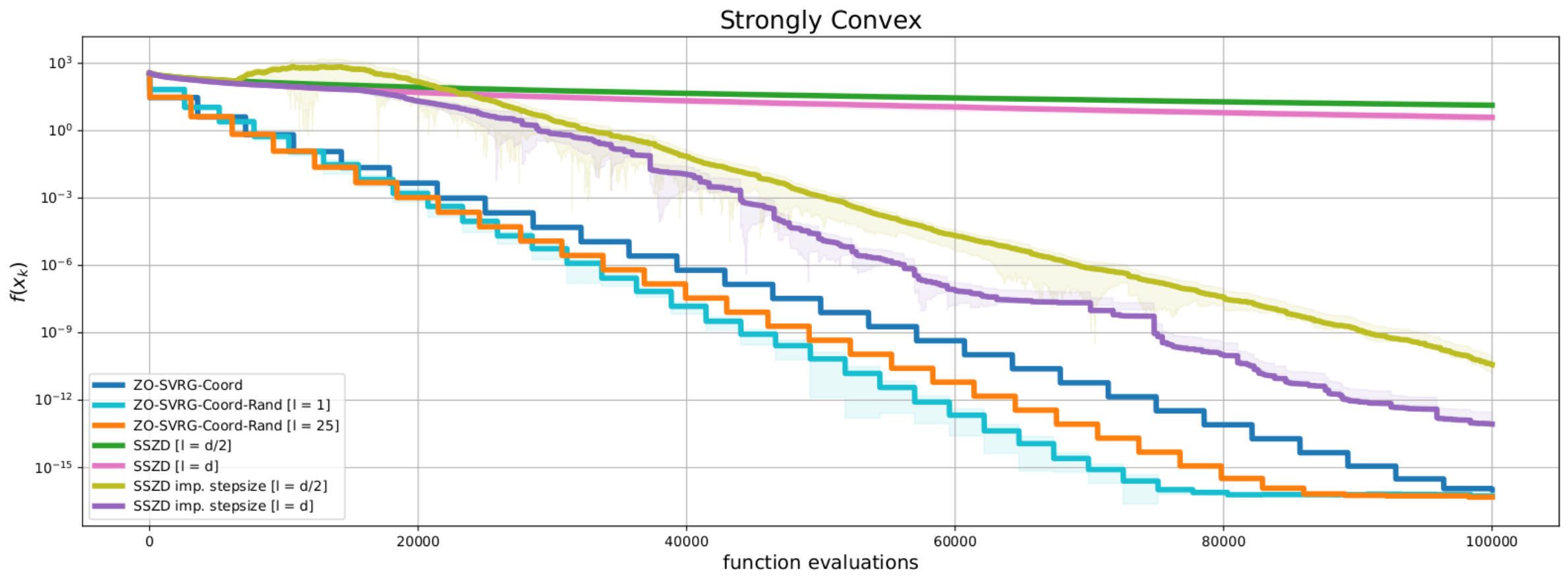}
    \includegraphics[width=0.8\linewidth]{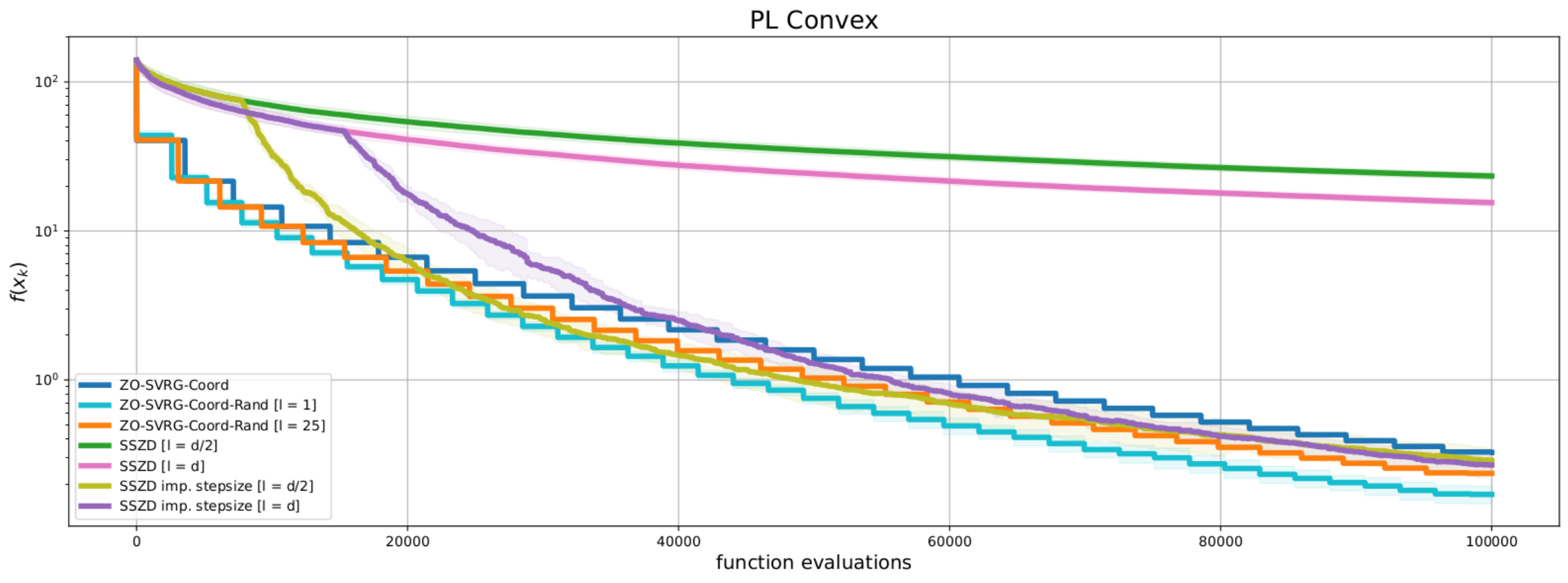}
    \includegraphics[width=0.8\linewidth]{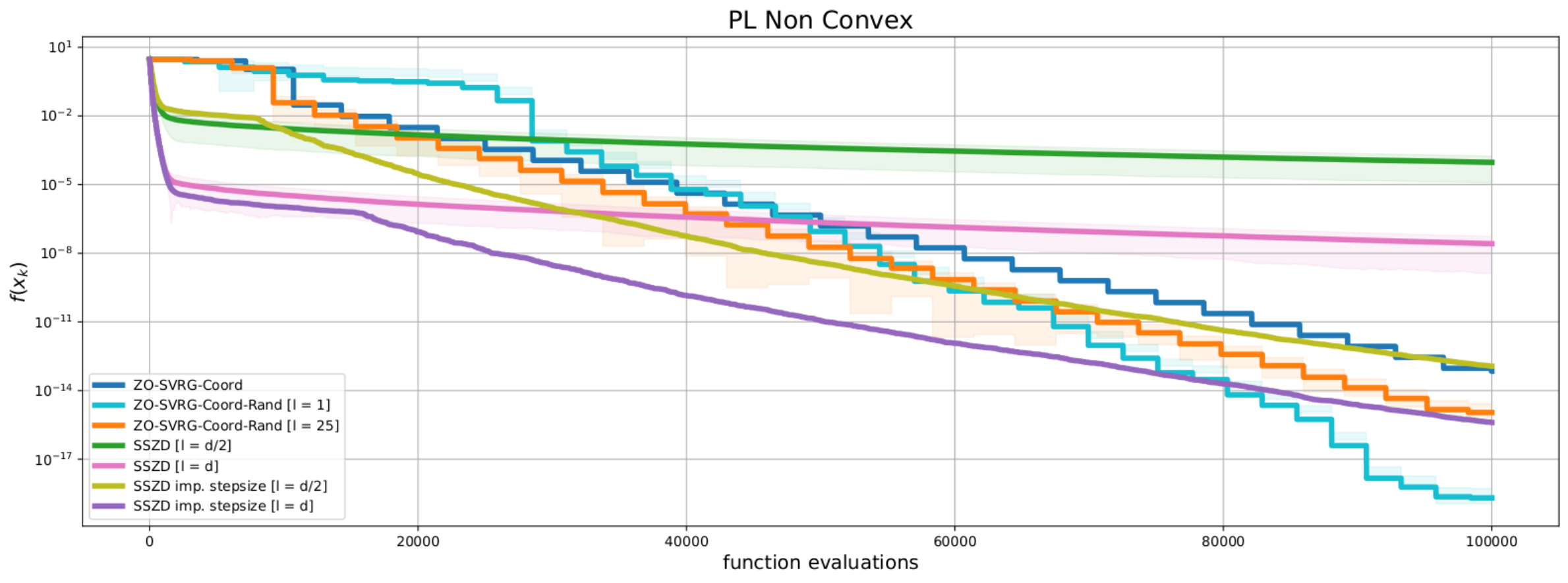}    
    \caption{Comparison of SSZD with zeroth order variance reduced methods.}
    \label{fig:vr_comp}
\end{figure}
\noindent In Figure \ref{fig:vr_comp}, we observe that both ZO-SVRG-Coord and ZO-SVRG-Coord-Rand perform better than S-SZD, as expected. However, these two algorithms are applicable only for the finite sum of functions while S-SZD can be used for facing the more general problem \eqref{eqn:problem}. Moreover, we observe that the step size proposed in eq. \eqref{app:imp_step} empirically improves the performance of S-SZD. However, such a choice for the stepsize is an heuristic and introduces more parameters to tune. Notably, in the PL convex setting, S-SZD achieves results similar to variance-reduced methods.

\subsection{Strategies to build direction matrices}\label{app:dir_gen}
We show here two examples of how to generate direction matrices $P_k$ for Algorithm~\ref{alg:szo}. Specifically, we can have \textit{coordinate} directions or \textit{spherical} directions.\\
{\it Coordinate Directions.} The easiest way to build a direction matrix which satisfies Assumption~\ref{ass:pk_exp} consists in taking the identity matrix $I \in \R^{d \times d}$ and sample $l$ columns uniformly at random (without replacement). Then for each sampled column, we flip the sign (we multiply by $-1$) with probability $1/2$ and we multiply the matrix by $\sqrt{d/l}$.\\
{\it Spherical Directions.} A different way to build a direction matrix satisfying Assumption~\ref{ass:pk_exp}, consists in taking directions orthogonal and uniform on a sphere \cite{sph_orth}. The algorithm consists in generating a matrix $Z \in \R^{d \times d}$ s.t. $Z_{i,j} \sim \mathcal{N}(0, 1)$ and computing the QR-decomposition. 
Then, the matrix $Q$ is \textit{truncated} by taking the first $0 < l \leq d$ directions and multiplied by $\sqrt{d/l}$.

\section{Limitations}\label{app:limitations}
In this appendix, we discuss the main limitations of Algorithm~\ref{alg:szo}. A limitation of S-SZD consists in requiring multiple function evaluations to perform steps. Indeed, in many contexts, evaluating the function can be very time-expensive (e.g., running simulations). In such contexts, to get good performances in time, we would need to reduce $l$. However, as we observe in Section \ref{sec:experiements}, we get worse performances. %
Another limitation is in the choice of the parameters. Indeed, in order to achieve convergence, we cannot choose constant step-size and discretization parameters but we have to tune sequence of parameters. To address this limitation, one approach is to introduce and investigate adaptive parameter selection methods, such as line-search algorithms. Moreover, in many practical scenarios, the stochastic objective function cannot be evaluated at different $x$ with the same $z$ (e.g., bandit feedback \cite{frazier2018tutorial}). %
Also note that if the sequence $h_k$ decreases too fast we can have numerical instability in computing \eqref{eqn:surrogate}.

\bibliographystyle{plain}
\bibliography{bibliography}%

\begin{thebibliography}{10}

\bibitem{ds_noise}
Edward~J. Anderson and Michael~C. Ferris.
\newblock A direct search algorithm for optimization with noisy function
  evaluations.
\newblock {\em SIAM Journal on Optimization}, 11(3):837--857, 2001.

\bibitem{Audet2021}
Charles Audet, Kwassi~Joseph Dzahini, Michael Kokkolaras, and S{\'e}bastien
  Le~Digabel.
\newblock Stochastic mesh adaptive direct search for blackbox optimization
  using probabilistic estimates.
\newblock {\em Computational Optimization and Applications}, 79(1):1--34, May
  2021.

\bibitem{balasubramanian2018zeroth}
Krishnakumar Balasubramanian and Saeed Ghadimi.
\newblock Zeroth-order (non)-convex stochastic optimization via conditional
  gradient and gradient updates.
\newblock {\em Advances in Neural Information Processing Systems}, 31, 2018.

\bibitem{berahas2022theoretical}
Albert~S Berahas, Liyuan Cao, Krzysztof Choromanski, and Katya Scheinberg.
\newblock A theoretical and empirical comparison of gradient approximations in
  derivative-free optimization.
\newblock {\em Foundations of Computational Mathematics}, 22(2):507--560, 2022.

\bibitem{stp}
El~Houcine Bergou, Eduard Gorbunov, and Peter Richtárik.
\newblock Stochastic three points method for unconstrained smooth minimization.
\newblock {\em SIAM Journal on Optimization}, 30(4):2726--2749, 2020.

\bibitem{blanchet2019convergence}
Jose Blanchet, Coralia Cartis, Matt Menickelly, and Katya Scheinberg.
\newblock Convergence rate analysis of a stochastic trust-region method via
  supermartingales.
\newblock {\em INFORMS journal on optimization}, 1(2):92--119, 2019.

\bibitem{bolte}
Jérôme Bolte, Aris Daniilidis, Olivier Ley, and Laurent Mazet.
\newblock Characterizations of Łojasiewicz inequalities: Subgradient flows,
  talweg, convexity.
\newblock {\em Transactions of The American Mathematical Society - TRANS AMER
  MATH SOC}, 362:3319--3363, 06 2009.

\bibitem{ZO-BCD}
Hanqin Cai, Yuchen Lou, Daniel Mckenzie, and Wotao Yin.
\newblock A zeroth-order block coordinate descent algorithm for huge-scale
  black-box optimization.
\newblock In {\em Proceedings of the 38th International Conference on Machine
  Learning}, volume 139 of {\em Proceedings of Machine Learning Research},
  pages 1193--1203, 18--24 Jul 2021.

\bibitem{cai2022one}
HanQin Cai, Daniel McKenzie, Wotao Yin, and Zhenliang Zhang.
\newblock A one-bit, comparison-based gradient estimator.
\newblock {\em Applied and Computational Harmonic Analysis}, 60:242--266, 2022.

\bibitem{zoro}
HanQin Cai, Daniel McKenzie, Wotao Yin, and Zhenliang Zhang.
\newblock Zeroth-order regularized optimization (zoro): Approximately sparse
  gradients and adaptive sampling.
\newblock {\em SIAM Journal on Optimization}, 32(2):687--714, 2022.

\bibitem{Cartis2023}
Coralia Cartis and Lindon Roberts.
\newblock Scalable subspace methods for derivative-free nonlinear least-squares
  optimization.
\newblock {\em Mathematical Programming}, 199(1):461--524, May 2023.

\bibitem{stars}
Ruobing Chen and Stefan Wild.
\newblock Randomized derivative-free optimization of noisy convex functions.
\newblock 2015.

\bibitem{chikuse2012statistics}
Yasuko Chikuse.
\newblock {\em Statistics on special manifolds}, volume 174.
\newblock 2012.

\bibitem{pmlr-v80-choromanski18a}
Krzysztof Choromanski, Mark Rowland, Vikas Sindhwani, Richard Turner, and
  Adrian Weller.
\newblock Structured evolution with compact architectures for scalable policy
  optimization.
\newblock In Jennifer Dy and Andreas Krause, editors, {\em Proceedings of the
  35th International Conference on Machine Learning}, volume~80 of {\em
  Proceedings of Machine Learning Research}, pages 970--978, 10--15 Jul 2018.

\bibitem{chung}
K.~L. Chung.
\newblock On a stochastic approximation method.
\newblock {\em The Annals of Mathematical Statistics}, 25(3):463--483, 1954.

\bibitem{Conn2009IntroductionTD}
Andrew~R. Conn, Katya Scheinberg, and Lu{\'i}s~Nunes Vicente.
\newblock Introduction to derivative-free optimization.
\newblock In {\em MPS-SIAM series on optimization}, 2009.

\bibitem{Dodangeh2016}
M.~Dodangeh and L.~N. Vicente.
\newblock Worst case complexity of direct search under convexity.
\newblock {\em Mathematical Programming}, 155(1):307--332, Jan 2016.

\bibitem{3f5c9a64536d47be8f46be0b369d8c51}
M.~Dodangeh, {L. N.} Vicente, and Zaikun Zhang.
\newblock On the optimal order of worst case complexity of direct search.
\newblock {\em Optimization Letters}, 10(4):699--708, April 2016.

\bibitem{Dua:2019}
Dheeru Dua and Casey Graff.
\newblock {UCI} machine learning repository, 2017.

\bibitem{duchi_smoothing}
John~C. Duchi, Peter~L. Bartlett, and Martin~J. Wainwright.
\newblock Randomized smoothing for stochastic optimization.
\newblock {\em SIAM Journal on Optimization}, 22(2):674--701, 2012.

\bibitem{sto_md}
John~C. Duchi, Michael~I. Jordan, Martin~J. Wainwright, and Andre Wibisono.
\newblock Optimal rates for zero-order convex optimization: The power of two
  function evaluations.
\newblock {\em IEEE Transactions on Information Theory}, 61(5):2788--2806,
  2015.

\bibitem{dzahini2024directsearchstochasticoptimization}
K.~J. Dzahini and S.~M. Wild.
\newblock Direct search for stochastic optimization in random subspaces with
  zeroth-, first-, and second-order convergence and expected complexity, 2024.

\bibitem{Dzahini2022}
Kwassi~Joseph Dzahini.
\newblock Expected complexity analysis of stochastic direct-search.
\newblock {\em Computational Optimization and Applications}, 81(1):179--200,
  Jan 2022.

\bibitem{dzahini2022stochastic}
Kwassi~Joseph Dzahini and Stefan~M. Wild.
\newblock Stochastic trust-region algorithm in random subspaces with
  convergence and expected complexity analyses, 2022.

\bibitem{flaxman2005online}
Abraham Flaxman, Adam~Tauman Kalai, and Brendan McMahan.
\newblock Online convex optimization in the bandit setting: Gradient descent
  without a gradient.
\newblock In {\em SODA '05 Proceedings of the sixteenth annual ACM-SIAM
  symposium on Discrete algorithms}, pages 385--394, January 2005.

\bibitem{cbo}
Massimo Fornasier, Timo Klock, and Konstantin Riedl.
\newblock Consensus-based optimization methods converge globally, 2022.

\bibitem{frazier2018tutorial}
Peter~I. Frazier.
\newblock A tutorial on bayesian optimization.
\newblock 2018.

\bibitem{8161093}
R.~Garmanjani and L.~N. Vicente.
\newblock Smoothing and worst-case complexity for direct-search methods in
  nonsmooth optimization.
\newblock {\em IMA Journal of Numerical Analysis}, 33(3):1008--1028, 2013.

\bibitem{garrigos2023handbook}
Guillaume Garrigos and Robert~M. Gower.
\newblock Handbook of convergence theorems for (stochastic) gradient methods.
\newblock 2024.

\bibitem{pmlr-v162-gasnikov22a}
Alexander Gasnikov, Anton Novitskii, Vasilii Novitskii, Farshed Abdukhakimov,
  Dmitry Kamzolov, Aleksandr Beznosikov, Martin Takac, Pavel Dvurechensky, and
  Bin Gu.
\newblock The power of first-order smooth optimization for black-box non-smooth
  problems.
\newblock In {\em Proceedings of the 39th International Conference on Machine
  Learning}, volume 162 of {\em Proceedings of Machine Learning Research},
  pages 7241--7265, Virtual Conference, 17--23 Jul 2022. PMLR.

\bibitem{ghadimi_lam}
Saeed Ghadimi and Guanghui Lan.
\newblock Stochastic first- and zeroth-order methods for nonconvex stochastic
  programming.
\newblock {\em SIAM Journal on Optimization}, 23(4):2341--2368, 2013.

\bibitem{grapiglia2022worst}
Geovani~Nunes Grapiglia.
\newblock Worst-case evaluation complexity of a derivative-free quadratic
  regularization method.
\newblock 2022.

\bibitem{prob_ds}
S.~Gratton, C.~W. Royer, L.~N. Vicente, and Z.~Zhang.
\newblock Direct search based on probabilistic descent.
\newblock {\em SIAM Journal on Optimization}, 25(3):1515--1541, 2015.

\bibitem{ha2024iteration}
Yunsoo Ha and Sara Shashaani.
\newblock Iteration complexity and finite-time efficiency of adaptive sampling
  trust-region methods for stochastic derivative-free optimization, 2024.

\bibitem{stars_acc}
Jordan~R. Hall and Varis Carey.
\newblock Accelerating derivative-free optimization with dimension reduction
  and hyperparameter learning.
\newblock 2021.

\bibitem{cma}
Nikolaus Hansen.
\newblock {\em The CMA Evolution Strategy: A Comparing Review}, volume 192,
  pages 75--102.
\newblock 06 2007.

\bibitem{harris2020array}
Charles~R. Harris, K.~Jarrod Millman, St{\'{e}}fan~J. van~der Walt, Ralf
  Gommers, Pauli Virtanen, David Cournapeau, Eric Wieser, Julian Taylor,
  Sebastian Berg, Nathaniel~J. Smith, Robert Kern, Matti Picus, Stephan Hoyer,
  Marten~H. van Kerkwijk, Matthew Brett, Allan Haldane, Jaime~Fern{\'{a}}ndez
  del R{\'{i}}o, Mark Wiebe, Pearu Peterson, Pierre G{\'{e}}rard-Marchant,
  Kevin Sheppard, Tyler Reddy, Warren Weckesser, Hameer Abbasi, Christoph
  Gohlke, and Travis~E. Oliphant.
\newblock Array programming with {NumPy}.
\newblock {\em Nature}, 585(7825):357--362, September 2020.

\bibitem{KieWol52}
J.~Wolfowitz J.~Kiefer.
\newblock Stochastic estimation of the maximum of a regression function.
\newblock {\em The Annals of Mathematical Statistics}, 23(3):462--–466, 1952.

\bibitem{pmlr-v97-ji19a}
Kaiyi Ji, Zhe Wang, Yi~Zhou, and Yingbin Liang.
\newblock Improved zeroth-order variance reduced algorithms and analysis for
  nonconvex optimization.
\newblock In Kamalika Chaudhuri and Ruslan Salakhutdinov, editors, {\em
  Proceedings of the 36th International Conference on Machine Learning},
  volume~97 of {\em Proceedings of Machine Learning Research}, pages
  3100--3109, 09--15 Jun 2019.

\bibitem{khaled2020better}
Ahmed Khaled and Peter Richtárik.
\newblock Better theory for sgd in the nonconvex world, 2020.

\bibitem{conv_dds}
Sujin Kim and Dali Zhang.
\newblock Convergence properties of direct search methods for stochastic
  optimization.
\newblock pages 1003--1011, 12 2010.

\bibitem{comp_search}
Tamara~G. Kolda, Robert~Michael Lewis, and Virginia Torczon.
\newblock Optimization by direct search: New perspectives on some classical and
  modern methods.
\newblock {\em SIAM Review}, 45(3):385--482, 2003.

\bibitem{sds}
Jakub Konečný and Peter Richtárik.
\newblock Simple complexity analysis of simplified direct search, 2014.

\bibitem{Kozk2021ASS}
David Koz{\'a}k, Stephen Becker, Alireza Doostan, and Luis Tenorio.
\newblock A stochastic subspace approach to gradient-free optimization in high
  dimensions.
\newblock {\em Comput. Optim. Appl.}, 79:339--368, 2021.

\bibitem{kozak2021zeroth}
David Koz{\'a}k, Cesare Molinari, Lorenzo Rosasco, Luis Tenorio, and Silvia
  Villa.
\newblock Zeroth-order optimization with orthogonal random directions, May
  2023.

\bibitem{NEURIPS2018_ba9a56ce}
Sijia Liu, Bhavya Kailkhura, Pin-Yu Chen, Paishun Ting, Shiyu Chang, and Lisa
  Amini.
\newblock Zeroth-order stochastic variance reduction for nonconvex
  optimization.
\newblock In S.~Bengio, H.~Wallach, H.~Larochelle, K.~Grauman, N.~Cesa-Bianchi,
  and R.~Garnett, editors, {\em Advances in Neural Information Processing
  Systems}, volume~31, 2018.

\bibitem{lojasiewicz1963topological}
Stanislaw Lojasiewicz.
\newblock A topological property of real analytic subsets.
\newblock {\em Coll. du CNRS, Les {\'e}quations aux d{\'e}riv{\'e}es
  partielles}, 117(87-89):2, 1963.

\bibitem{htru2}
R.~J. Lyon, B.~W. Stappers, S.~Cooper, J.~M. Brooke, and J.~D. Knowles.
\newblock {Fifty years of pulsar candidate selection: from simple filters to a
  new principled real-time classification approach}.
\newblock {\em Monthly Notices of the Royal Astronomical Society},
  459(1):1104--1123, 04 2016.

\bibitem{10.5555/3326943.3327109}
Horia Mania, Aurelia Guy, and Benjamin Recht.
\newblock Simple random search of static linear policies is competitive for
  reinforcement learning.
\newblock In {\em Proceedings of the 32nd International Conference on Neural
  Information Processing Systems}, NIPS'18, page 1805–1814, Red Hook, NY,
  USA, 2018. Curran Associates Inc.

\bibitem{falkonhopt2022}
Giacomo Meanti, Luigi Carratino, Ernesto De~Vito, and Lorenzo Rosasco.
\newblock Efficient hyperparameter tuning for large scale kernel ridge
  regression.
\newblock In {\em Proceedings of The 25th International Conference on
  Artificial Intelligence and Statistics}, 2022.

\bibitem{falkonlibrary2020}
Giacomo Meanti, Luigi Carratino, Lorenzo Rosasco, and Alessandro Rudi.
\newblock Kernel methods through the roof: Handling billions of points
  efficiently.
\newblock In H.~Larochelle, M.~Ranzato, R.~Hadsell, M.~F. Balcan, and H.~Lin,
  editors, {\em Advances in Neural Information Processing Systems}, volume~33,
  pages 14410--14422, Red Hook, NY, USA, 2020. Curran Associates, Inc.

\bibitem{sph_orth}
Francesco Mezzadri.
\newblock How to generate random matrices from the classical compact groups.
\newblock {\em Notices of the American Mathematical Society}, 54:592–604, 10
  2006.

\bibitem{nest}
Yurii Nesterov and Vladimir Spokoiny.
\newblock Random gradient-free minimization of convex functions.
\newblock {\em Found. Comput. Math.}, 17(2):527–566, apr 2017.

\bibitem{NEURIPS2019_9015}
Adam Paszke, Sam Gross, Francisco Massa, Adam Lerer, James Bradbury, Gregory
  Chanan, Trevor Killeen, Zeming Lin, Natalia Gimelshein, Luca Antiga, Alban
  Desmaison, Andreas Kopf, Edward Yang, Zachary DeVito, Martin Raison, Alykhan
  Tejani, Sasank Chilamkurthy, Benoit Steiner, Lu~Fang, Junjie Bai, and Soumith
  Chintala.
\newblock Py{{T}}orch: An imperative style, high-performance deep learning
  library.
\newblock In {\em Advances in Neural Information Processing Systems 32}, pages
  8024--8035. Curran Associates, Inc., Red Hook, NY, USA, 2019.

\bibitem{scikit-learn}
F.~Pedregosa, G.~Varoquaux, A.~Gramfort, V.~Michel, B.~Thirion, O.~Grisel,
  M.~Blondel, P.~Prettenhofer, R.~Weiss, V.~Dubourg, J.~Vanderplas, A.~Passos,
  D.~Cournapeau, M.~Brucher, M.~Perrot, and E.~Duchesnay.
\newblock Scikit-learn: Machine learning in {P}ython.
\newblock {\em Journal of Machine Learning Research}, 12:2825--2830, 2011.

\bibitem{polyak1987introduction}
Boris~T Polyak.
\newblock Introduction to optimization.
\newblock {\em Optimization Software Inc., Publications Division, New York},
  1:32, 1987.

\bibitem{popovic2004direct}
Dobrivoje Popovic and Andrew~R Teel.
\newblock Direct search methods for nonsmooth optimization.
\newblock In {\em 2004 43rd IEEE Conference on Decision and Control (CDC)(IEEE
  Cat. No. 04CH37601)}, volume~3, pages 3173--3178. IEEE, 2004.

\bibitem{Powell1994}
M.~J.~D. Powell.
\newblock {\em A Direct Search Optimization Method That Models the Objective
  and Constraint Functions by Linear Interpolation}, pages 51--67.
\newblock Springer Netherlands, Dordrecht, 1994.

\bibitem{Powell2002}
M.~J.~D. Powell.
\newblock Uobyqa: unconstrained optimization by quadratic approximation.
\newblock {\em Mathematical Programming}, 92(3):555--582, May 2002.

\bibitem{Powell2006}
M.~J.~D. Powell.
\newblock {\em The NEWUOA software for unconstrained optimization without
  derivatives}, pages 255--297.
\newblock Springer US, Boston, MA, 2006.

\bibitem{powell2009bobyqa}
Michael~JD Powell et~al.
\newblock The bobyqa algorithm for bound constrained optimization without
  derivatives.
\newblock {\em Cambridge NA Report NA2009/06, University of Cambridge,
  Cambridge}, 26:26--46, 2009.

\bibitem{price2006direct}
Christopher~John Price, M~Reale, and BL~Robertson.
\newblock A direct search method for smooth and nonsmooth unconstrained
  optimization.
\newblock {\em ANZIAM Journal}, 48:C927--C948, 2006.

\bibitem{adabkb}
Marco Rando, Luigi Carratino, Silvia Villa, and Lorenzo Rosasco.
\newblock Ada-bkb: Scalable gaussian process optimization on continuous domains
  by adaptive discretization.
\newblock In {\em Proceedings of The 25th International Conference on
  Artificial Intelligence and Statistics}, volume 151 of {\em Proceedings of
  Machine Learning Research}, pages 7320--7348, Virtual Conference, 28--30 Mar
  2022. PMLR.

\bibitem{ozd}
Marco Rando, Cesare Molinari, Lorenzo Rosasco, and Silvia Villa.
\newblock An optimal structured zeroth-order algorithm for non-smooth
  optimization.
\newblock In A.~Oh, T.~Naumann, A.~Globerson, K.~Saenko, M.~Hardt, and
  S.~Levine, editors, {\em Advances in Neural Information Processing Systems},
  volume~36, pages 36738--36767, 2023.

\bibitem{ROBBINS1971233}
H.~Robbins and D.~Siegmund.
\newblock A convergence theorem for non negative almost supermartingales and
  some applications.
\newblock In Jagdish~S. Rustagi, editor, {\em Optimizing Methods in
  Statistics}, pages 233--257. Academic Press, Academic Press, 1971.

\bibitem{prob_ds_sketched}
Lindon Roberts and Cl\'{e}ment~W. Royer.
\newblock Direct search based on probabilistic descent in reduced spaces, 2023.

\bibitem{rudi2017falkon}
Alessandro Rudi, Luigi Carratino, and Lorenzo Rosasco.
\newblock Falkon: An optimal large scale kernel method.
\newblock In {\em Advances in Neural Information Processing Systems},
  volume~30, 2017.

\bibitem{Salgia2021ADB}
Sudeep Salgia, Sattar Vakili, and Qing Zhao.
\newblock A domain-shrinking based bayesian optimization algorithm with
  order-optimal regret performance.
\newblock In {\em NeurIPS}, 2021.

\bibitem{Salimans2017EvolutionSA}
Tim Salimans, Jonathan Ho, Xi~Chen, Szymon Sidor, and Ilya Sutskever.
\newblock Evolution strategies as a scalable alternative to reinforcement
  learning.
\newblock 2017.

\bibitem{shekhar2018gaussian}
Shubhanshu Shekhar and Tara Javidi.
\newblock {{G}aussian process bandits with adaptive discretization}.
\newblock {\em Electronic Journal of Statistics}, 12(2):3829 -- 3874, 2018.

\bibitem{pso}
Dr.~Narinder Singh.
\newblock Review of particle swarm optimization.
\newblock {\em International Journal of Computational Intelligence and
  Information Security, April 2012}, Vol. 3:34--44, 01 2012.

\bibitem{spall}
James~C. Spall.
\newblock {\em Introduction to Stochastic Search and Optimization}.
\newblock John Wiley \& Sons, Inc., USA, 1 edition, 2003.

\bibitem{srinivas2009gaussian}
Niranjan Srinivas, Andreas Krause, Sham~M Kakade, and Matthias Seeger.
\newblock {G}aussian process optimization in the bandit setting: No regret and
  experimental design.
\newblock In {\em Proceedings of the 27th International Conference on
  International Conference on Machine Learning}, pages 1015--1022, 2010.

\bibitem{totzeck2022trends}
Claudia Totzeck.
\newblock Trends in consensus-based optimization.
\newblock In Nicola Bellomo, Jos{\'e}~Antonio Carrillo, and Eitan Tadmor,
  editors, {\em Active Particles, Volume 3: Advances in Theory, Models, and
  Applications}, pages 201--226. Springer International Publishing, Cham, 2022.

\bibitem{stief_zeroth}
Tianyu Wang and Yasong Feng.
\newblock Convergence rates of zeroth order gradient descent for Łojasiewicz
  functions.
\newblock {\em INFORMS Journal on Computing}, 0(0):null, 0.

\end{thebibliography}

\end{document}